\documentclass[11pt]{article}
\usepackage{amsmath,amsfonts,amssymb, amsthm,enumerate,graphicx}
\usepackage{tikz,authblk}
\usepackage{subfig}
\usepackage{url}
\newtheorem{theorem} {{\textsf{Theorem}}}
\newtheorem{proposition}[theorem]{{\textsf{Proposition}}}

\newtheorem{remark}[theorem]{{\textsf{Remark}}}
\newtheorem{example}[theorem]{{\textsf{Example}}}
\newtheorem{lemma}[theorem]{{\textsf{Lemma}}}
\newtheorem{conjecture}[theorem]{{\textsf{Conjecture}}}

\begin{document}
\title{Regular genus and gem-complexity of some mapping tori}
\author{Biplab Basak}
\date{}
\maketitle
\vspace{-10mm}
\begin{center}

\noindent {\small Department of Mathematics, Indian Institute of Technology Delhi, Hauz Khaz, New Delhi 110016, India.}

\noindent {\small {\em E-mail address:} \url{biplab8654@gmail.com}}

\medskip

\date{January 25, 2019}
\end{center}
\hrule

\begin{abstract}
In this article, we construct a crystallization of the mapping torus of some (PL) homeomorphisms $f:M \to M$ for a certain class of PL-manifolds $M$. These yield upper bounds for gem-complexity and regular genus of a large class of PL-manifolds. The bound for the regular genus is sharp for the mapping torus of some (PL) homeomorphisms $f:M \to M$, where $M$ is $\mathbb{RP}^2$, $\mathbb{RP}^2\#\mathbb{RP}^2$, $\mathbb{S}^1\times \mathbb{S}^1$, $\mathbb{RP}^3$, $\mathbb{S}^{2} \times \mathbb{S}^1$, $\mathbb{S}^{\hspace{.2mm}2} \mbox{$\times
\hspace{-2.6mm}_{-}$} \, \mathbb{S}^{\hspace{.1mm}1}$ or $\mathbb{S}^d$. In particular, for $M=\mathbb{S}^{d-1} \times \mathbb{S}^1$ or $\mathbb{S}^{\hspace{.2mm}d-1} \mbox{$\times\hspace{-2.6mm}_{-}$} \, \mathbb{S}^{\hspace{.1mm}1}$, our construction gives a crystallization of a mapping torus of a (PL) homeomorphism $f:M \to M$ with regular genus $d^2-d$. As a consequence, we prove the existence of an orientable mapping torus of a (PL) homeomorphism $f:(\mathbb{S}^{2} \times \mathbb{S}^1)\to (\mathbb{S}^{2} \times \mathbb{S}^1)$  with regular genus 6. This disproves a conjecture of Spaggiari which states that regular genus six characterizes the topological product $\mathbb{RP}^3 \times \mathbb{S}^1$ among closed connected prime orientable PL $4$-manifolds.
\end{abstract}

\noindent {\small {\em MSC 2010\,:} Primary 57Q15. Secondary 05C15; 57N10; 57N13; 57N15; 57Q05; 55R10.

}

\medskip

\hrule

\section{Introduction and Results}
First, let us recall that, a {\em fiber bundle} is a topological space which is locally a product space. Let $M$ be a closed connected manifold. For a homeomorphism $f:M \to M$, let  $M_f=\frac{M \times [0,1]}{(x,0)\sim (f(x),1)}$. Then $M_f$ is an $M$-fiber bundle over $\mathbb{S}^1$. The bundle $M_f$ is called {\em the mapping torus of the homeomorphism $f:M\to M$}. In this article, we have mainly focused on PL $d$-manifolds $M$ and constructed mapping tori of some (PL) homeomorphisms $f:M\to M$.

A crystallization of a PL $d$-manifold is a contracted colored graph which represents the manifold (for details and related notations see Subsection \ref{crystal}).
The existence of crystallizations for every closed connected PL $d$-manifold is ensured by a classical theorem due to Pezzana (see \cite{pe74}, and \cite{bcg01,fgg86} for subsequent generalizations).  In dimension three there is a different version of crystallization for compact connected 3-manifolds without spherical boundary components (cf. \cite{crm16}). Moreover, interesting connections between crystallization theory and random tensor models (as a possible approach to study of quantum gravity) have been found in \cite{ccdg18}. Further, quite recently, this theory has been extended to PL pseudo-manifolds (including PL $d$-manifolds) of arbitrary dimensions (cf. \cite{ccg18,cg18}).

Given a PL $d$-manifold $M$, its {\em gem-complexity} is the non-negative integer $\mathit{k}(M):=p-1$, where $2p$ is the minimum order of a crystallization of $M$. 
Let $(\Gamma, \gamma)$ be a $(d+1)$-colored graph (i.e., at each vertex of the graph, there are exactly  $d+1$ edges with different colors from the color set $\Delta_d:=\{0, \dots, d\}$). An embedding $i : \Gamma \hookrightarrow F$ of $\Gamma$ into a closed surface $F$ is called {\it regular} if there exists a cyclic permutation $\varepsilon=(\varepsilon_0, \varepsilon_1, \dots, \varepsilon_d)$ of the color set $\Delta_d$, such that the boundary of each face of $i(\Gamma)$ is a bi-colored cycle with edges alternately colored by $\varepsilon_j, \varepsilon_{j+1}$ for each $j \in \Delta_d$ (the addition is modulo $d+1$). Then,
the {\it regular genus} $\rho (\Gamma)$ of $(\Gamma,\gamma)$ is the least genus (resp., half the genus) of an
orientable (resp., non-orientable) surface into which $\Gamma$ embeds regularly, and the {\it regular genus} $\mathcal G (M)$ of a closed connected PL $d$-manifold $M$ is defined as the minimum regular genus of its crystallizations. Note that the notion of regular genus extends classical notions to arbitrary dimension. In fact, the regular genus of a closed connected orientable (resp., non-orientable) surface coincides with its genus (resp., half its genus), while the regular genus of a closed connected 3-manifold coincides with its Heegaard genus (see  \cite{ga81}).

The invariant regular genus (resp., gem-complexity) has been intensively studied, yielding some important general results. For example, regular genus (resp., gem-complexity) zero characterizes the $d$-sphere among all closed connected PL $d$-manifolds (cf. \cite{[FG$_2$]}). In \cite{bd14}, the authors give a lower bound for the gem-complexity of all closed 3-manifolds. A lower bound for the regular genus (resp., gem-complexity) is also obtained for all closed PL 4-manifolds in \cite{bc15}. Moreover, many authors give some bounds for the regular genus (resp., gem-complexity) of some PL $d$-manifolds (cf. \cite{cs7,cr4,fg82,gg93,sp99}). In this article, our constructions give upper bounds for the regular genus and gem-complexity of some mapping tori. More precisely, we prove the following.

\begin{theorem} \label{theorem:genus1}
If $M=\mathbb{S}^{d-1}\times \mathbb{S}^1$ (resp., $\mathbb{S}^{\hspace{.2mm}d-1} \mbox{$\times
\hspace{-2.6mm}_{-}$} \, \mathbb{S}^{\hspace{.1mm}1}$) then there exists an orientable (resp., a non-orientable) mapping torus $M_f$ of a (PL) homeomorphism $f:M\to M$ with the following properties.

$$\mathcal{G}(M_f)\leq d^2-3 \mbox{ and }\mathit{k}(M_f)\leq d^2+3d+1.$$

\noindent Moreover, if $d=3$, then the regular genus $\mathcal{G}(M_f)$ of $M_f$ is $6$.
\end{theorem}

\begin{remark}\label{remark:conjecture}
{\rm Topological classification of closed connected PL $4$-manifolds according to the regular genus is a classical problem in combinatorial topology (cf. \cite{cav99, cm93}). A PL classification of closed connected PL $4$-manifolds according to gem-complexity (up to gem-complexity 8) has been studied in \cite{cc16}. But the solutions of these problems are known for a fairly small set of manifolds. The difficulty of the problem increases especially with the growth of the  regular genus (resp., gem-complexity). In fact, the topological classification of closed connected orientable PL $4$-manifolds according to the regular genus is open when the regular genus is greater than 5.}
\end{remark}

In \cite{sp99}, Spaggiari conjectured the following. 
 
\begin{conjecture}\label{conjecture}
The topological product $\mathbb{RP}^3 \times \mathbb{S}^1$ is the unique, up to
topological homeomorphism, closed connected prime orientable PL $4$-manifold of genus six.
\end{conjecture}

Since in Theorem \ref{theorem:genus1}, we have obtained an orientable mapping torus of a (PL) homeomorphism $f:\mathbb{S}^2 \times \mathbb{S}^1\to \mathbb{S}^2\times \mathbb{S}^1$ with regular genus 6, the above conjecture is false.

\begin{theorem}\label{theorem:genus2}
For $q\geq 2$, there exists an orientable mapping torus $L(q,1)_f$ of a (PL) homeomorphism $f:L(q,1) \to L(q,1)$ such that $\mathcal{G}(L(q,1)_f) \leq 5q-4$ and $\mathit{k}(L(q,1)_f) \leq 10q-1$.

In particular, there exists an orientable (resp., a non-orientable) mapping torus $L(2,1)_f$ (resp., $L(2,1)_{\tilde f}$) of a  (PL) homeomorphism $f:L(2,1)\to L(2,1)$ (resp., $\tilde f:L(2,1)\to L(2,1)$) with regular genus 6.
\end{theorem}

For $g,h\geq 1$, let $T_g$ denote the orientable surface $\#_g(\mathbb{S}^1\times \mathbb{S}^1)$ and $U_h$ denote the non-orientable surface  $\#_h\mathbb{RP}^2$ of genus $g$ and $h$ respectively.

\begin{theorem} \label{theorem:genus3}
If $M$ is $T_1 \times \mathbb{S}^1$ or $\mathbb{RP}^d$ for some $d \geq 3$ odd, then there exists an orientable (resp., a non-orientable) mapping torus $M_f$ (resp., $M_{\tilde f}$) of a  (PL) homeomorphism $f:M\to M$ (resp., $\tilde f:M\to M$) with the following properties.
\begin{enumerate}[$(i)$]
\item $\mathcal{G}((T_1 \times \mathbb{S}^1)_f),\mathcal{G}((T_1 \times \mathbb{S}^1)_{\tilde f}) \leq 16$ and $\mathit{k}((T_1 \times \mathbb{S}^1)_f),\mathit{k}((T_1 \times \mathbb{S}^1)_{\tilde f}) \leq 59$.
\item $\mathcal{G}((\mathbb{RP}^d)_f),\mathcal{G}((\mathbb{RP}^d)_{\tilde f})  \leq 1+ 2^{d-3}(d^2-4)$ and $\mathit{k}((\mathbb{RP}^d)_f),\mathit{k}((\mathbb{RP}^d)_{\tilde f}) \leq (d+2)2^{d-1}-1$.
\end{enumerate}
\end{theorem}

\begin{theorem} \label{theorem:genus4}
If $M$ is $U_h\times \mathbb{S}^1$ or $\mathbb{RP}^d$ for $1 \leq h \leq 2$ and $d \geq 4$ even,  then there exists a non-orientable mapping torus  $M_{\tilde f}$ of a (PL) homeomorphism $\tilde f:M\to M$ with the following properties.
\begin{enumerate}[$(i)$]

\item $\mathcal{G}((U_h \times \mathbb{S}^1)_{\tilde f}) \leq 5h+6$ and $\mathit{k}((U_h \times \mathbb{S}^1)_{\tilde f}) \leq 20h+19$.
\item $\mathcal{G}((\mathbb{RP}^d)_{\tilde f})  \leq 1+ 2^{d-3}(d^2-4)$ and $\mathit{k}((\mathbb{RP}^d)_{\tilde f}) \leq (d+2)2^{d-1}-1$.
\end{enumerate}
\end{theorem}

\section{Preliminaries}  \label{sec:prelim}

\subsection{Crystallization} \label{crystal}

Crystallization theory is a representation method for the whole class of piecewise linear (PL) manifolds, without restrictions about dimension, connectedness, orientability or boundary properties.  A graph is called {\em $(d+1)$-regular} if the number of edges adjacent to each vertex is $(d+1)$. For $n\geq 2$, by $kC_n$ we mean a graph consists of $k$ disjoint $n$-cycles. The disjoint union of the graphs $G$ and $H$ is denoted by $G \sqcup H$. We refer to \cite{bm08} for standard terminology on graphs, and to \cite{bj84} for CW-complexes and related notions.

A {\it $(d+1)$-colored graph} is a pair
$(\Gamma,\gamma)$, where $\Gamma= (V(\Gamma),$ $E(\Gamma))$ is a regular multigraph (i.e. multiple edges are allowed, while loops are forbidden) of degree $d+1$ and the surjective map $\gamma : E(\Gamma) \to \Delta_d:=\{0,1, \dots , d\}$ is a proper edge-coloring (i.e. $\gamma(e) \ne \gamma(f)$ for any pair $e,f$ of adjacent edges). The elements of the set $\Delta_d$ are called the {\it colors} of
$\Gamma$. 

For each $B \subseteq \Delta_d$ with $h$ elements, then the
graph $\Gamma_B =(V(\Gamma), \gamma^{-1}(B))$ is an $h$-colored graph with edge-coloring $\gamma|_{\gamma^{-1}(B)}$.
If $\Gamma_{\Delta_d \setminus\{c\}}$ is connected for all $c\in \Delta_d$, then  $(\Gamma,\gamma)$ is called {\em contracted}. For a color set $\{i_1,i_2,\dots,i_k\} \subset \Delta_d$, $\Gamma_{\{i_1,i_2, \dots, i_k\}}$ denotes the subgraph restricted to the color set  $\{i_1,i_2,\dots,i_k\}$  and $g_{\{i_1,i_2, \dots, i_k\}}$ denotes the number of connected components of the graph $\Gamma_{\{i_1, i_2, \dots, i_k\}}$. 

\smallskip

By dropping the regularity condition in the definition of $(d+1)$-colored graph, one
obtains the notion of $(d+1)$-colored graph with boundary. A  boundary vertex is simply a
vertex of degree less than $d+1$. A $(d+1)$-colored graph $(\Gamma,\gamma)$ with boundary is said to be {\em regular with respect to the color $c$} if $\Gamma_{\Delta_{d}\setminus\{c\}}$ is regular of degree $d$. For such a graph $(\Gamma,\gamma)$  we can define its  boundary graph $(\partial \Gamma,\partial \gamma)$ as follows:

\begin{itemize}
\item{} $V(\partial \Gamma)$ is in bijection with the set of boundary vertices of $\Gamma$.

\item{} $u,v \in V(\partial \Gamma)$ are joined in $\partial \Gamma$ by an edge of color $i$ if and only if $u$ and $v$ are joined in $\Gamma$  by a path with edges alternately colored by $i$ and $c$.
\end{itemize} 

\smallskip

Each $(d+1)$-colored graph (resp., $(d+1)$-colored graph with boundary) uniquely determines a $d$-dimensional simplicial cell-complex ${\mathcal K}(\Gamma)$, which is said to be {\it associated to $\Gamma$}:

\begin{itemize}
\item{} for every vertex $v\in V(\Gamma)$, take a $d$-simplex $\sigma(v)$ and label injectively its $d+1$ vertices by the colors of $\Delta_d$,

\item{} for every  edge of color $i$ between $v,w\in V(\Gamma)$, identify the ($d-1$)-faces of $\sigma(v)$ and $\sigma(w)$ opposite to $i$-labeled vertices, so that equally labeled vertices coincide.
\end{itemize}

If the geometrical carrier $|{\mathcal K}(\Gamma)|$ is PL homeomorphic to a PL $d$-manifold $M$, then the $(d+1)$-colored graph (resp., $(d+1)$-colored graph with boundary) $(\Gamma,\gamma)$ is said to {\it represent} $M$. If $(\Gamma,\gamma)$ represents a closed PL $d$-manifold $M$ and is contracted, then it is called a {\it crystallization} of $M$; in this case the number of vertices of ${\mathcal K}(\Gamma)$ is exactly $d+1$. It is not hard to see that  $|{\mathcal K}(\Gamma)|$ is orientable if and only if $\Gamma$ is a bipartite graph. A PL $d$-manifold with boundary can always be represented by a $(d+1)$-colored graph $(\Gamma,\gamma)$ with boundary, which is regular with respect to a fixed color $c$, for some $c\in \Delta_{d}$. Thus, we can define its boundary-graph $(\partial \Gamma,\partial \gamma)$, and each component of the boundary-graph $(\partial \Gamma,\partial \gamma)$ represents a component of $\partial(M)$.

\smallskip

Let $(\Gamma,\gamma)$  be a $(d+1)$-colored graph with color set $\Delta_d$. Then $I(\Gamma):=(I_V,I_c):\Gamma \to \Gamma$ is called an {\em isomorphism} if $I_V: V(\Gamma) \to V(\Gamma)$ and $I_c:\Delta_d \to \Delta_d$ are bijective maps such that $uv$ is an edge of color $i \in \Delta_d$ if and only if $I_V(u)I_V(v)$ is an edge of color $I_c(i) \in \Delta_d$. Observe that the isomorphism $I(\Gamma):\Gamma \to \Gamma$ naturally induces a map $I({\mathcal K}(\Gamma)):{\mathcal K}(\Gamma) \to {\mathcal K}(\Gamma)$ which sends each vertex labeled by the color $i$ to the vertex labeled by the colors $I_c(i)$ and sends each $d$-simplex $\sigma(v)$ to the $d$-simplex $\sigma(I_V(v))$. By the definition of isomorphism, the $(d-1)$-faces of the $d$-simplices $\sigma(u)$ and $\sigma(v)$ opposite to the vertex labeled by the color $i$ are identified in $\mathcal K$ if and only if the $(d-1)$-faces of the $d$-simplices $\sigma(I_V(u))$ and $\sigma(I_V(v))$ opposite to the vertex labeled by the color $I_c(i)$ are identified in $\mathcal K$. Thus, by the construction, $|I({\mathcal K}(\Gamma))|:|{\mathcal K}(\Gamma)| \to |{\mathcal K}(\Gamma)|$ is a PL homeomorphism.   

\subsection{Fundamental group}\label{fundamental}
Let  $(\Gamma, \gamma)$ be a crystallization (with the color set $\Delta_d$) of a connected closed $d$-manifold $M$. So,
$\Gamma$ is a $(d+1)$-regular graph. Choose two colors, say, $i$ and $j$ from $\Delta_d$. Let $\{G_1, \dots, G_{s+1}\}$
be the set of all connected components of $\Gamma_{\Delta_d\setminus \{i,j\}}$ and $\{H_1, \dots, H_{t+1}\}$ be the set
of all connected components of $\Gamma_{\{i,j\}}$. Since $\Gamma$ is regular, each $H_p$ is an even cycle. Note
that, if $d=2$, then $\Gamma_{\{i,j\}}$ is connected and hence $H_1= \Gamma_{\{i,j\}}$. Take a set $\widetilde{S} =
\{x_1, \dots, x_s, x_{s+1}\}$ of $s+1$ elements. Choose a vertex $v_1$ in $H_k$. Let $H_k = v_1 e_{1}^i v_2 e_{2}^j v_3 e_{3}^i
v_{4} \cdots e_{2l-1}^i v_{2l}e_{2l}^jv_1$, where $ e_{p}^i$ and  $ e_{q}^j$ are edges with colors $i$ and $j$
respectively. Define $
\tilde{r}_k := x_{k_2}^{+1} x_{k_3}^{-1}x_{k_4}^{+1}  \cdots
x_{k_{2l}}^{+1}x_{k_1}^{-1}$,
where $G_{k_h}$ is the component of $\Gamma_{\Delta_d\setminus \{i,j\}}$ containing $v_h$.  For $1\leq k\leq t+1$, let
$r_k$ be the word obtained from $\tilde{r}_k$ by deleting $x_{s+1}^{\pm 1}$'s in $\tilde{r}_k$. In \cite{ga79b}, Gagliardi proved the following.

\begin{proposition} \label{prop:gagliardi79b}
For $d\geq 2$, let  $(\Gamma, \gamma)$ be a crystallization of a connected closed $d$-manifold $M$. For two
colors $i, j$, let $s$, $t$, $x_p$, $r_q$ be as above. If $\pi_1(M)$ is the fundamental group of $M$, then
$$
\pi_1(M) \cong \left\{ \begin{array}{lcl}
\langle {x_1, x_2,\dots, x_s} ~ | ~ {r_1} \rangle & \mbox{if}
& d=2,   \\
\langle {x_1, x_2, \dots, x_s} ~ | ~ {r_1, \dots, r_t} \rangle
& \mbox{if} & d\geq 3.
\end{array}\right.
$$
\end{proposition}

\subsection{Regular Genus of PL $d$-manifolds}\label{sec:genus}

As already briefly recalled in Section 1, the notion of {\it regular genus} is strictly related to the existence of {\it regular embeddings} of crystallizations into closed surfaces, i.e., embeddings whose regions are bounded by the images of bi-colored cycles, with colors consecutive in a fixed permutation of the color set.
More precisely, according to \cite{ga81},  if $(\Gamma, \gamma)$ is a crystallization of an orientable (resp., non-orientable) PL $d$-manifold $M$ ($d \geq 2$), then for a fixed cyclic permutation  $\varepsilon:= (\varepsilon_0, \varepsilon_1, \varepsilon_2, \dots , \varepsilon_d)$ of $\Delta_d$,  a regular embedding $i_\varepsilon : \Gamma \hookrightarrow F_\varepsilon$ exists,  where $F_{\varepsilon}$ is the closed orientable (resp., non-orientable) surface with Euler characteristic
\begin{eqnarray}\label{relation_chi}
\chi_{\varepsilon}(\Gamma)= \sum_{i \in \mathbb{Z}_{d+1}}g_{\{\varepsilon_i,\varepsilon_{i+1}\}} + (1-d) \ \frac{\#V(\Gamma)}{2}.
\end{eqnarray}
In the orientable (resp., non-orientable) case, the integer
$$\rho_{\varepsilon}(\Gamma) = 1 - \chi_{\varepsilon}(\Gamma)/2$$
is equal to the genus (resp,. half the genus) of the surface $F_{\varepsilon}$.
Then, the regular genus $\rho (\Gamma)$ of $(\Gamma,\gamma)$  and the regular genus $\mathcal G (M)$ of $M$ are:
$$\rho(\Gamma)= \min \{\rho_{\varepsilon}(\Gamma) \ | \  \varepsilon \ \text{ is a cyclic permutation of } \ \Delta_d\};$$
$$\mathcal G(M) = \min \{\rho(\Gamma) \ | \  (\Gamma,\gamma) \mbox{ is a crystallization of } M\}.$$

From \cite{bc15}, we know the following.
\begin{proposition} \label{prop:lowerbound}
Let $M$ be a (closed connected) PL $4$-manifold with $rk(\pi_1(M))=m.$
Then,
$$ \mathit{k}(M) \ \ge \ 3 \chi (M) + 10m -6 \mbox{  and  } \mathcal G(M) \ \ge \ 2 \chi (M) + 5m -4.$$
\end{proposition}

For a closed connected 3-manifold $N$, if $M$ is a mapping torus of a (PL) homeomorphism $f:N\to N$, i.e., if $M$ is an $N$-bundle over $\mathbb{S}^1$ then $\chi(M)=0$. Therefore, $ \mathit{k}(M) \geq 10m -6$ and $\mathcal G(M)\geq 5m -4$.

\section{Proofs of the Results}
\begin{lemma} \label{lemma:1}
For $d \geq 1$, let $(\Gamma,\gamma)$ be a $2p$-vertex crystallization (with color set $\Delta_d$) of a closed connected $d$-manifold $M$. If there exist isomorphisms $I(\Gamma):=(I_V,I_c):\Gamma \to \Gamma$ with the property $I_c(i)=i+1$ for $i \in \Delta_d$ (addition is modulo $d+1$) then there exists a crystallization $(\bar \Gamma,\bar \gamma)$ of a  mapping torus $M_f$ of a (PL) homeomorphism $f:M\to M$ with $2(d+2)p$ vertices. 
\end{lemma}
\begin{proof}
Let $I(\Gamma):=(I_V,I_c):\Gamma \to \Gamma$ be an isomorphism with the property $I_c(i)=i+1$ (addition is modulo $d+1$) for $i \in \Delta_d$. Let the vertices of the crystallization $(\Gamma,\gamma)$ be $x^{(0)},x^{(1)},\dots, x^{(2p-1)}$. Now, we construct a crystallization $(\bar \Gamma,\bar \gamma)$ of a  mapping torus $M_f$ of a (PL) homeomorphism $f:M\to M$ with $2(d+2)p$ vertices. First, we construct a $(d+2)$-colored graph $(\Gamma',\gamma')$ with boundary (with color set $\Delta_{d+1}$) by the following steps.

\begin{enumerate}[$(i)$]
\item For each $k \in \{0,1,\dots,d+1\}$, consider a new graph $\Gamma^{k}$ obtained from $\Gamma_{\{0,1,\dots,d-1\}}$ by renaming the vertices $x^{(j)}$ as $x^{(j)}_k$ for $0 \leq j \leq 2p-1$ and the colors $i$ as $i+k$ (addition is modulo $d+2$). Let $I_V(x^{(j)}_k)$ be the vertices in $\Gamma^{k}$ corresponding to the vertices $I_V(x^{(j)})$.

\item For each $k \in \{0,1,\dots,d\}$, join the vertices $I_V(x^{(j)}_k)$ and $x^{(j)}_{k+1}$ by an edge of color $k+d+1$ (addition is modulo $d+2$) for $0 \leq j \leq 2p-1$.
\end{enumerate}  

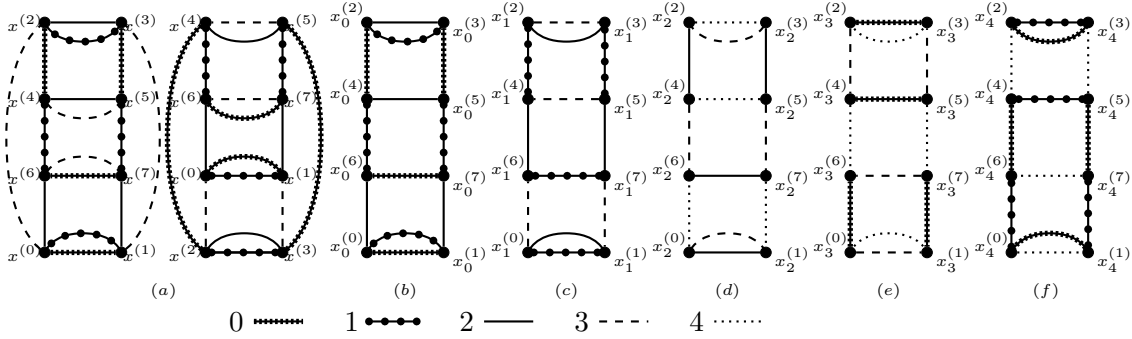
\begin{figure}[ht]
\tikzstyle{vert}=[circle, draw, fill=black!100, inner sep=0pt, minimum width=4pt]
\tikzstyle{vertex}=[circle, draw, fill=black!00, inner sep=0pt, minimum width=4pt]
\tikzstyle{ver}=[]
\tikzstyle{extra}=[circle, draw, fill=black!50, inner sep=0pt, minimum width=2pt]
\tikzstyle{edge} = [draw,thick,-]
\centering
\begin{tikzpicture}[scale=0.51]
\begin{scope}[shift={(0,0)}]
\foreach \w/\x/\y/\z in {-1.5/-1/3/0,1.5/1/3/1,1.5/1/1/2,-1.5/-1/1/3,-1.5/-1/-1/4,1.5/1/-1/5,1.5/1/-3/6,-1.5/-1/-3/7}{
\node[vert] (a\z) at (\x,\y){};
%\node[ver] () at (\w,\y){\tiny{$x^{(\z)}$}};
}
\foreach \w/\x/\y/\z in
{-1.5/-1/3/2,1.5/1/3/3,-1.5/1/1/4,1.5/-1/1/5,-1.5/-1/-1/6,1.5/1/-1/7,-1.5/1/-3/0,1.5/-1/-3/1}
{\node[ver] () at (\w,\y){\tiny{$x^{(\z)}$}};}

\foreach \x/\y in {a2/a1,a0/a3,a7/a6,a5/a4}{
\draw [line width=2pt, line cap=rectengle, dash pattern=on 1pt off 1]  (\x) -- (\y);
\draw [edge]  (\x) -- (\y);}

\foreach \x/\y in {a0/a1,a2/a3,a5/a6,a7/a4,a3/a4,a2/a5}{
\draw [edge]  (\x) -- (\y);}

\foreach \x/\y in {a3/a4,a2/a5}{
\draw [line width=3pt, line cap=round, dash pattern=on 0pt off 2\pgflinewidth]  (\x) -- (\y);}

\draw[line width=3pt, line cap=round, dash pattern=on 0pt off 2\pgflinewidth]  plot [smooth,tension=1.5] coordinates{(a0)(0,2.5)(a1)};
\draw[line width=3pt, line cap=round, dash pattern=on 0pt off 2\pgflinewidth]  plot [smooth,tension=1.5] coordinates{(a6)(0,-2.5)(a7)};
\draw[edge]  plot [smooth,tension=1.5] coordinates{(a0)(0,2.5)(a1)};
\draw[edge]  plot [smooth,tension=1.5] coordinates{(a6)(0,-2.5)(a7)};
\draw[edge, dashed]  plot [smooth,tension=1.5] coordinates{(a2)(0,0.5)(a3)};
\draw[edge, dashed]  plot [smooth,tension=1.5] coordinates{(a5)(0,-0.5)(a4)};
\draw[edge, dashed]  plot [smooth,tension=1.5] coordinates{(a0)(-2,0)(a7)};
\draw[edge, dashed]  plot [smooth,tension=1.5] coordinates{(a1)(2,0)(a6)};
\end{scope}

\begin{scope}[shift={(4.2,0)}]
\foreach \w/\x/\y/\z in {-1.5/-1/3/3,1.5/1/3/2,1.5/1/1/5,-1.5/-1/1/4,-1.5/-1/-1/7,1.5/1/-1/6,1.5/1/-3/1,-1.5/-1/-3/0}{
\node[vert] (a\z) at (\x,\y){};}
%\node[ver] () at (\w,\y){\tiny{$x^{(\z)}$}};}
\foreach \w/\x/\y/\z in
{-1.5/-1/3/4,1.5/1/3/5,-1.5/1/1/6,1.5/-1/1/7,-1.5/-1/-1/0,1.5/1/-1/1,
-1.5/1/-3/2,1.5/-1/-3/3}
{\node[ver] () at (\w,\y){\tiny{$x^{(\z)}$}};}

\foreach \x/\y in {a7/a6,a0/a1,a3/a4,a2/a5}{
\draw [line width=3pt, line cap=round, dash pattern=on 0pt off 2\pgflinewidth]  (\x) -- (\y);
\draw [edge]  (\x) -- (\y);}

\foreach \x/\y in {a0/a1,a5/a6,a7/a4}{
\draw [edge]  (\x) -- (\y);}

\foreach \x/\y in {a0/a7,a1/a6,a2/a3,a5/a4}{
\draw [edge, dashed]  (\x) -- (\y);}

\draw[edge]  plot [smooth,tension=1.5] coordinates{(a0)(0,-2.5)(a1)};
\draw[edge]  plot [smooth,tension=1.5] coordinates{(a3)(0,2.5)(a2)};
\draw[edge]  plot [smooth,tension=1.5] coordinates{(a4)(0,0.5)(a5)};
\draw[edge]  plot [smooth,tension=1.5] coordinates{(a6)(0,-0.5)(a7)};
\draw[line width=2pt, line cap=rectengle, dash pattern=on 1pt off 1]   plot [smooth,tension=1.5] coordinates{(a4)(0,0.5)(a5)};
\draw[line width=2pt, line cap=rectengle, dash pattern=on 1pt off 1]   plot [smooth,tension=1.5] coordinates{(a6)(0,-0.5)(a7)};
\draw[edge]  plot [smooth,tension=1.5] coordinates{(a0)(-2,0)(a3)};
\draw[edge]  plot [smooth,tension=1.5] coordinates{(a1)(2,0)(a2)};
\draw[line width=2pt, line cap=rectengle, dash pattern=on 1pt off 1] plot [smooth,tension=1.5] coordinates{(a0)(-2,0)(a3)};
\draw[line width=2pt, line cap=rectengle, dash pattern=on 1pt off 1]  plot [smooth,tension=1.5] coordinates{(a1)(2,0)(a2)};
\end{scope}

\begin{scope}[shift={(8.4,0)}]
\foreach \w/\x/\y/\z in {-1.5/-1/3/0,1.5/1/3/1,1.5/1/1/2,-1.5/-1/1/3,-1.5/-1/-1/4,1.5/1/-1/5,1.5/1/-3/6,-1.5/-1/-3/7}{
\node[vert] (a\z) at (\x,\y){};}

\node[ver] () at (-1.5,3.2){\tiny{$ x_{0}^{(2)}$}};
\node[ver] () at (1.7,2.8){\tiny{$ x_{0}^{(3)}$}};
\node[ver] () at (-1.5,1.2){\tiny{$ x_{0}^{(4)}$}};
\node[ver] () at (1.7,.8){\tiny{$ x_{0}^{(5)}$}};
\node[ver] () at (-1.5,-.8){\tiny{$ x_{0}^{(6)}$}};
\node[ver] () at (1.7,-1.2){\tiny{$ x_{0}^{(7)}$}};
\node[ver] () at (-1.5,-2.8){\tiny{$ x_{0}^{(0)}$}};
\node[ver] () at (1.7,-3.2){\tiny{$ x_{0}^{(1)}$}};

\foreach \x/\y in {a2/a1,a0/a3,a7/a6,a5/a4}{
\draw [line width=2pt, line cap=rectengle, dash pattern=on 1pt off 1]  (\x) -- (\y);
\draw [edge]  (\x) -- (\y);}

\foreach \x/\y in {a0/a1,a2/a3,a5/a6,a7/a4,a3/a4,a2/a5}{
\draw [edge]  (\x) -- (\y);}

\foreach \x/\y in {a3/a4,a2/a5}{
\draw [line width=3pt, line cap=round, dash pattern=on 0pt off 2\pgflinewidth]  (\x) -- (\y);}

\draw[line width=3pt, line cap=round, dash pattern=on 0pt off 2\pgflinewidth]  plot [smooth,tension=1.5] coordinates{(a0)(0,2.5)(a1)};
\draw[line width=3pt, line cap=round, dash pattern=on 0pt off 2\pgflinewidth]  plot [smooth,tension=1.5] coordinates{(a6)(0,-2.5)(a7)};
\draw[edge]  plot [smooth,tension=1.5] coordinates{(a0)(0,2.5)(a1)};
\draw[edge]  plot [smooth,tension=1.5] coordinates{(a6)(0,-2.5)(a7)};

\end{scope}

\begin{scope}[shift={(12.6,0)}]
\foreach \w/\x/\y/\z in {-1.5/-1/3/0,1.5/1/3/1,1.5/1/1/2,-1.5/-1/1/3,-1.5/-1/-1/4,1.5/1/-1/5,1.5/1/-3/6,-1.5/-1/-3/7}{
\node[vert] (b\z) at (\x,\y){};}

\node[ver] () at (-1.5,3.2){\tiny{$ x_{1}^{(2)}$}};
\node[ver] () at (1.7,2.8){\tiny{$ x_{1}^{(3)}$}};
\node[ver] () at (-1.5,1.2){\tiny{$ x_{1}^{(4)}$}};
\node[ver] () at (1.7,.8){\tiny{$ x_{1}^{(5)}$}};
\node[ver] () at (-1.5,-.8){\tiny{$ x_{1}^{(6)}$}};
\node[ver] () at (1.7,-1.2){\tiny{$ x_{1}^{(7)}$}};
\node[ver] () at (-1.5,-2.8){\tiny{$ x_{1}^{(0)}$}};
\node[ver] () at (1.7,-3.2){\tiny{$ x_{1}^{(1)}$}};

\foreach \x/\y in {b2/b1,b0/b3,b7/b6,b5/b4}{
\draw [line width=3pt, line cap=round, dash pattern=on 0pt off 2\pgflinewidth]  (\x) -- (\y);
\draw [edge]  (\x) -- (\y);}

\foreach \x/\y in {b3/b4,b2/b5}{
\draw [edge]  (\x) -- (\y);}

\foreach \x/\y in {b0/b1,b2/b3,b5/b6,b7/b4}{
\draw [edge, dashed]  (\x) -- (\y);}

\draw[edge]  plot [smooth,tension=1.5] coordinates{(b0)(0,2.5)(b1)};
\draw[edge]  plot [smooth,tension=1.5] coordinates{(b6)(0,-2.5)(b7)};
\end{scope}

\begin{scope}[shift={(16.8,0)}]
\foreach \w/\x/\y/\z in {-1.5/-1/3/0,1.5/1/3/1,1.5/1/1/2,-1.5/-1/1/3,-1.5/-1/-1/4,1.5/1/-1/5,1.5/1/-3/6,-1.5/-1/-3/7}{
\node[vert] (c\z) at (\x,\y){};}
\node[ver] () at (-1.5,3.2){\tiny{$ x_{2}^{(2)}$}};
\node[ver] () at (1.7,2.8){\tiny{$ x_{2}^{(3)}$}};
\node[ver] () at (-1.5,1.2){\tiny{$ x_{2}^{(4)}$}};
\node[ver] () at (1.7,.8){\tiny{$ x_{2}^{(5)}$}};
\node[ver] () at (-1.5,-.8){\tiny{$ x_{2}^{(6)}$}};
\node[ver] () at (1.7,-1.2){\tiny{$ x_{2}^{(7)}$}};
\node[ver] () at (-1.5,-2.8){\tiny{$ x_{2}^{(0)}$}};
\node[ver] () at (1.7,-3.2){\tiny{$ x_{2}^{(1)}$}};
\foreach \x/\y in {c0/c3,c1/c2,c5/c4,c6/c7}{
\draw [edge]  (\x) -- (\y);}

\foreach \x/\y in {c0/c1,c2/c3,c5/c6,c7/c4}{
\draw [edge, dotted]  (\x) -- (\y);}

\foreach \x/\y in {c3/c4,c2/c5}{
\draw [edge, dashed]  (\x) -- (\y);}

\draw[edge, dashed] plot [smooth,tension=1.5] coordinates{(c0)(0,2.5)(c1)};
\draw[edge, dashed]  plot [smooth,tension=1.5] coordinates{(c6)(0,-2.5)(c7)};
\end{scope}

\begin{scope}[shift={(21,0)}]
\foreach \w/\x/\y/\z in {-1.5/-1/3/0,1.5/1/3/1,1.5/1/1/2,-1.5/-1/1/3,-1.5/-1/-1/4,1.5/1/-1/5,1.5/1/-3/6,-1.5/-1/-3/7}{
\node[vert] (d\z) at (\x,\y){};}
\node[ver] () at (-1.5,3.2){\tiny{$ x_{3}^{(2)}$}};
\node[ver] () at (1.7,2.8){\tiny{$ x_{3}^{(3)}$}};
\node[ver] () at (-1.5,1.2){\tiny{$ x_{3}^{(4)}$}};
\node[ver] () at (1.7,.8){\tiny{$ x_{3}^{(5)}$}};
\node[ver] () at (-1.5,-.8){\tiny{$ x_{3}^{(6)}$}};
\node[ver] () at (1.7,-1.2){\tiny{$ x_{3}^{(7)}$}};
\node[ver] () at (-1.5,-2.8){\tiny{$ x_{3}^{(0)}$}};
\node[ver] () at (1.7,-3.2){\tiny{$ x_{3}^{(1)}$}};
\foreach \x/\y in {d2/d1,d0/d3,d7/d6,d5/d4}{\draw [edge, dashed]  (\x) -- (\y);}

\foreach \x/\y in {d0/d1,d2/d3,d5/d6,d7/d4}{
\draw [edge]  (\x) -- (\y);
\draw [line width=2pt, line cap=rectengle, dash pattern=on 1pt off 1]  (\x) -- (\y);
}

\foreach \x/\y in {d3/d4,d2/d5}{
\draw [edge, dotted]  (\x) -- (\y);}

\draw[edge, dotted]  plot [smooth,tension=1.5] coordinates{(d0)(0,2.5)(d1)};
\draw[edge, dotted]  plot [smooth,tension=1.5] coordinates{(d6)(0,-2.5)(d7)};
\end{scope}

\begin{scope}[shift={(25.2,0)}]
\foreach \w/\x/\y/\z in {-1.5/-1/3/0,1.5/1/3/1,1.5/1/1/2,-1.5/-1/1/3,-1.5/-1/-1/4,1.5/1/-1/5,1.5/1/-3/6,-1.5/-1/-3/7}{
\node[vert] (e\z) at (\x,\y){};}
\node[ver] () at (-1.5,3.2){\tiny{$ x_{4}^{(2)}$}};
\node[ver] () at (1.7,2.8){\tiny{$ x_{4}^{(3)}$}};
\node[ver] () at (-1.5,1.2){\tiny{$ x_{4}^{(4)}$}};
\node[ver] () at (1.7,.8){\tiny{$ x_{4}^{(5)}$}};
\node[ver] () at (-1.5,-.8){\tiny{$ x_{4}^{(6)}$}};
\node[ver] () at (1.7,-1.2){\tiny{$ x_{4}^{(7)}$}};
\node[ver] () at (-1.5,-2.8){\tiny{$ x_{4}^{(0)}$}};
\node[ver] () at (1.7,-3.2){\tiny{$ x_{4}^{(1)}$}};
\foreach \x/\y in {e2/e1,e0/e3,e7/e6,e5/e4}{\draw [edge, dotted]  (\x) -- (\y);}
\foreach \x/\y in {e0/e1,e2/e3,e5/e6,e7/e4}{
\draw [edge]  (\x) -- (\y);
\draw [line width=3pt, line cap=round, dash pattern=on 0pt off 2\pgflinewidth]  (\x) -- (\y);
}

\foreach \x/\y in {e3/e4,e2/e5}{
\draw [edge]  (\x) -- (\y);}

\foreach \x/\y in {e3/e4,e2/e5}{
\draw [line width=2pt, line cap=rectengle, dash pattern=on 1pt off 1]  (\x) -- (\y);
}

\draw[line width=2pt, line cap=rectengle, dash pattern=on 1pt off 1]  plot [smooth,tension=1.5] coordinates{(e0)(0,2.5)(e1)};
\draw[line width=2pt, line cap=rectengle, dash pattern=on 1pt off 1]  plot [smooth,tension=1.5] coordinates{(e6)(0,-2.5)(e7)};
\draw[edge]  plot [smooth,tension=1.5] coordinates{(e0)(0,2.5)(e1)};
\draw[edge]  plot [smooth,tension=1.5] coordinates{(e6)(0,-2.5)(e7)};
\end{scope}

\begin{scope}[shift={(8,0.2)}]
\node[ver] (308) at (-4,-5){$0$};
\node[ver] (300) at (-1,-5){$1$};
\node[ver] (301) at (2,-5){$2$};
\node[ver] (302) at (5,-5){$3$};
\node[ver] (303) at (8,-5){$4$};
\node[ver] (309) at (-2,-5){};
\node[ver] (304) at (1,-5){};
\node[ver] (305) at (4,-5){};
\node[ver] (306) at (7,-5){};
\node[ver] (307) at (10,-5){};
\path[edge] (300) -- (304);
\path[edge] (308) -- (309);
\draw [line width=2pt, line cap=rectengle, dash pattern=on 1pt off 1]  (308) -- (309);
\draw [line width=3pt, line cap=round, dash pattern=on 0pt off 2\pgflinewidth]  (300) -- (304);
\path[edge] (301) -- (305);
\path[edge, dashed] (302) -- (306);
\path[edge, dotted] (303) -- (307);
\end{scope}

\node[ver] () at (2.1,-4){\tiny{$(a)$}};
\node[ver] () at (8.4,-4){\tiny{$(b)$ }};
\node[ver] () at (12.6,-4){\tiny{$(c)$ }};
\node[ver] () at (16.8,-4){\tiny{$(d)$ }};
\node[ver] () at (21,-4){\tiny{$(e)$ }};
\node[ver] () at (25.1,-4){\tiny{$(f)$ }};

\end{tikzpicture}
\caption{$(a)$ A crystallization $(\Gamma,\gamma)$ of $\mathbb{S}^2 \times \mathbb{S}^1$ and an isomorphism $I(\Gamma):\Gamma \to \Gamma$ such that $I_V(x^{(p)})=x^{(p+2)}$ (addition is modulo $8$) and $I_c(i)=i+1$ (addition is modulo $4$), $(b)~ \Gamma^{0}$, $(c)~ \Gamma^{1}$, $(d)~ \Gamma^{2}$, $(e)~ \Gamma^{3}$, $(f)~ \Gamma^{4}$.} \label{fig:1}
\end{figure}

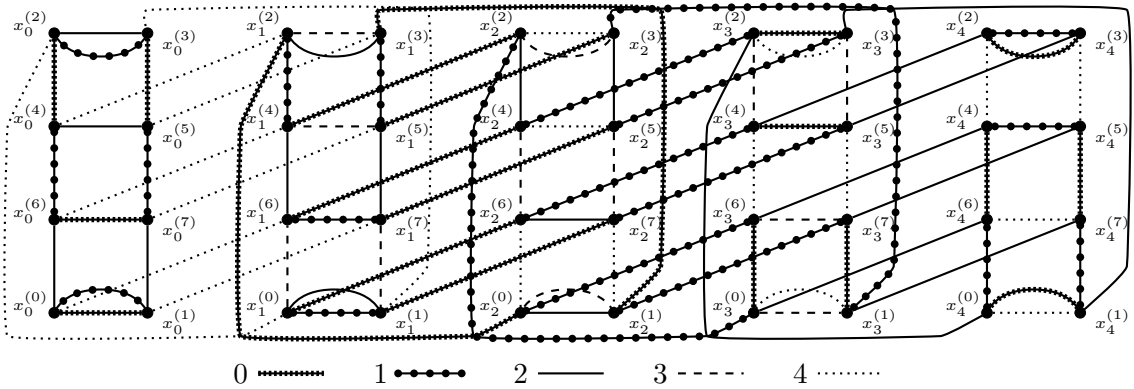
\begin{figure}[ht]
\tikzstyle{vert}=[circle, draw, fill=black!100, inner sep=0pt, minimum width=4pt]
\tikzstyle{vertex}=[circle, draw, fill=black!00, inner sep=0pt, minimum width=4pt]
\tikzstyle{ver}=[]
\tikzstyle{extra}=[circle, draw, fill=black!50, inner sep=0pt, minimum width=2pt]
\tikzstyle{edge} = [draw,thick,-]

\centering
\begin{tikzpicture}[scale=0.62]
\begin{scope}[shift={(0,0)}]
\foreach \w/\x/\y/\z in {-1.5/-1/3/0,1.5/1/3/1,1.5/1/1/2,-1.5/-1/1/3,-1.5/-1/-1/4,1.5/1/-1/5,1.5/1/-3/6,-1.5/-1/-3/7}{
\node[vert] (a\z) at (\x,\y){};}

\node[ver] () at (-1.5,3.2){\tiny{$ x_{0}^{(2)}$}};
\node[ver] () at (1.7,2.8){\tiny{$ x_{0}^{(3)}$}};
\node[ver] () at (-1.5,1.2){\tiny{$ x_{0}^{(4)}$}};
\node[ver] () at (1.7,.8){\tiny{$ x_{0}^{(5)}$}};
\node[ver] () at (-1.5,-.8){\tiny{$ x_{0}^{(6)}$}};
\node[ver] () at (1.7,-1.2){\tiny{$ x_{0}^{(7)}$}};
\node[ver] () at (-1.5,-2.8){\tiny{$ x_{0}^{(0)}$}};
\node[ver] () at (1.7,-3.2){\tiny{$ x_{0}^{(1)}$}};

\foreach \x/\y in {a2/a1,a0/a3,a7/a6,a5/a4}{
\draw [line width=2pt, line cap=rectengle, dash pattern=on 1pt off 1]  (\x) -- (\y);
\draw [edge]  (\x) -- (\y);}

\foreach \x/\y in {a0/a1,a2/a3,a5/a6,a7/a4,a3/a4,a2/a5}{
\draw [edge]  (\x) -- (\y);}

\foreach \x/\y in {a3/a4,a2/a5}{
\draw [line width=3pt, line cap=round, dash pattern=on 0pt off 2\pgflinewidth]  (\x) -- (\y);}

\draw[line width=3pt, line cap=round, dash pattern=on 0pt off 2\pgflinewidth]  plot [smooth,tension=1.5] coordinates{(a0)(0,2.5)(a1)};
\draw[line width=3pt, line cap=round, dash pattern=on 0pt off 2\pgflinewidth]  plot [smooth,tension=1.5] coordinates{(a6)(0,-2.5)(a7)};
\draw[edge]  plot [smooth,tension=1.5] coordinates{(a0)(0,2.5)(a1)};
\draw[edge]  plot [smooth,tension=1.5] coordinates{(a6)(0,-2.5)(a7)};

\end{scope}

\begin{scope}[shift={(5,0)}]
\foreach \w/\x/\y/\z in {-1.5/-1/3/0,1.5/1/3/1,1.5/1/1/2,-1.5/-1/1/3,-1.5/-1/-1/4,1.5/1/-1/5,1.5/1/-3/6,-1.5/-1/-3/7}{
\node[vert] (b\z) at (\x,\y){};}

\node[ver] () at (-1.5,3.2){\tiny{$ x_{1}^{(2)}$}};
\node[ver] () at (1.7,2.8){\tiny{$ x_{1}^{(3)}$}};
\node[ver] () at (-1.5,1.2){\tiny{$ x_{1}^{(4)}$}};
\node[ver] () at (1.7,.8){\tiny{$ x_{1}^{(5)}$}};
\node[ver] () at (-1.5,-.8){\tiny{$ x_{1}^{(6)}$}};
\node[ver] () at (1.7,-1.2){\tiny{$ x_{1}^{(7)}$}};
\node[ver] () at (-1.5,-2.8){\tiny{$ x_{1}^{(0)}$}};
\node[ver] () at (1.7,-3.2){\tiny{$ x_{1}^{(1)}$}};

\foreach \x/\y in {b2/b1,b0/b3,b7/b6,b5/b4}{
\draw [line width=3pt, line cap=round, dash pattern=on 0pt off 2\pgflinewidth]  (\x) -- (\y);
\draw [edge]  (\x) -- (\y);}

\foreach \x/\y in {b3/b4,b2/b5}{
\draw [edge]  (\x) -- (\y);}

\foreach \x/\y in {b0/b1,b2/b3,b5/b6,b7/b4}{
\draw [edge, dashed]  (\x) -- (\y);}

\draw[edge]  plot [smooth,tension=1.5] coordinates{(b0)(0,2.5)(b1)};
\draw[edge]  plot [smooth,tension=1.5] coordinates{(b6)(0,-2.5)(b7)};
\end{scope}

\begin{scope}[shift={(10,0)}]
\foreach \w/\x/\y/\z in {-1.5/-1/3/0,1.5/1/3/1,1.5/1/1/2,-1.5/-1/1/3,-1.5/-1/-1/4,1.5/1/-1/5,1.5/1/-3/6,-1.5/-1/-3/7}{
\node[vert] (c\z) at (\x,\y){};}
\node[ver] () at (-1.5,3.2){\tiny{$ x_{2}^{(2)}$}};
\node[ver] () at (1.7,2.8){\tiny{$ x_{2}^{(3)}$}};
\node[ver] () at (-1.5,1.2){\tiny{$ x_{2}^{(4)}$}};
\node[ver] () at (1.7,.8){\tiny{$ x_{2}^{(5)}$}};
\node[ver] () at (-1.5,-.8){\tiny{$ x_{2}^{(6)}$}};
\node[ver] () at (1.7,-1.2){\tiny{$ x_{2}^{(7)}$}};
\node[ver] () at (-1.5,-2.8){\tiny{$ x_{2}^{(0)}$}};
\node[ver] () at (1.7,-3.2){\tiny{$ x_{2}^{(1)}$}};
\foreach \x/\y in {c0/c3,c1/c2,c5/c4,c6/c7}{
\draw [edge]  (\x) -- (\y);}

\foreach \x/\y in {c0/c1,c2/c3,c5/c6,c7/c4}{
\draw [edge, dotted]  (\x) -- (\y);}

\foreach \x/\y in {c3/c4,c2/c5}{
\draw [edge, dashed]  (\x) -- (\y);}

\draw[edge, dashed] plot [smooth,tension=1.5] coordinates{(c0)(0,2.5)(c1)};
\draw[edge, dashed]  plot [smooth,tension=1.5] coordinates{(c6)(0,-2.5)(c7)};
\end{scope}

\begin{scope}[shift={(15,0)}]
\foreach \w/\x/\y/\z in {-1.5/-1/3/0,1.5/1/3/1,1.5/1/1/2,-1.5/-1/1/3,-1.5/-1/-1/4,1.5/1/-1/5,1.5/1/-3/6,-1.5/-1/-3/7}{
\node[vert] (d\z) at (\x,\y){};}
\node[ver] () at (-1.5,3.2){\tiny{$ x_{3}^{(2)}$}};
\node[ver] () at (1.7,2.8){\tiny{$ x_{3}^{(3)}$}};
\node[ver] () at (-1.5,1.2){\tiny{$ x_{3}^{(4)}$}};
\node[ver] () at (1.7,.8){\tiny{$ x_{3}^{(5)}$}};
\node[ver] () at (-1.5,-.8){\tiny{$ x_{3}^{(6)}$}};
\node[ver] () at (1.7,-1.2){\tiny{$ x_{3}^{(7)}$}};
\node[ver] () at (-1.5,-2.8){\tiny{$ x_{3}^{(0)}$}};
\node[ver] () at (1.7,-3.2){\tiny{$ x_{3}^{(1)}$}};
\foreach \x/\y in {d2/d1,d0/d3,d7/d6,d5/d4}{\draw [edge, dashed]  (\x) -- (\y);}

\foreach \x/\y in {d0/d1,d2/d3,d5/d6,d7/d4}{
\draw [edge]  (\x) -- (\y);
\draw [line width=2pt, line cap=rectengle, dash pattern=on 1pt off 1]  (\x) -- (\y);
}

\foreach \x/\y in {d3/d4,d2/d5}{
\draw [edge, dotted]  (\x) -- (\y);}

\draw[edge, dotted]  plot [smooth,tension=1.5] coordinates{(d0)(0,2.5)(d1)};
\draw[edge, dotted]  plot [smooth,tension=1.5] coordinates{(d6)(0,-2.5)(d7)};
\end{scope}

\begin{scope}[shift={(20,0)}]
\foreach \w/\x/\y/\z in {-1.5/-1/3/0,1.5/1/3/1,1.5/1/1/2,-1.5/-1/1/3,-1.5/-1/-1/4,1.5/1/-1/5,1.5/1/-3/6,-1.5/-1/-3/7}{
\node[vert] (e\z) at (\x,\y){};}
\node[ver] () at (-1.5,3.2){\tiny{$ x_{4}^{(2)}$}};
\node[ver] () at (1.7,2.8){\tiny{$ x_{4}^{(3)}$}};
\node[ver] () at (-1.5,1.2){\tiny{$ x_{4}^{(4)}$}};
\node[ver] () at (1.7,.8){\tiny{$ x_{4}^{(5)}$}};
\node[ver] () at (-1.5,-.8){\tiny{$ x_{4}^{(6)}$}};
\node[ver] () at (1.7,-1.2){\tiny{$ x_{4}^{(7)}$}};
\node[ver] () at (-1.5,-2.8){\tiny{$ x_{4}^{(0)}$}};
\node[ver] () at (1.7,-3.2){\tiny{$ x_{4}^{(1)}$}};
\foreach \x/\y in {e2/e1,e0/e3,e7/e6,e5/e4}{\draw [edge, dotted]  (\x) -- (\y);}
\foreach \x/\y in {e0/e1,e2/e3,e5/e6,e7/e4}{
\draw [edge]  (\x) -- (\y);
\draw [line width=3pt, line cap=round, dash pattern=on 0pt off 2\pgflinewidth]  (\x) -- (\y);
}

\foreach \x/\y in {e3/e4,e2/e5}{
\draw [edge]  (\x) -- (\y);}

\foreach \x/\y in {e3/e4,e2/e5}{
\draw [line width=2pt, line cap=rectengle, dash pattern=on 1pt off 1]  (\x) -- (\y);
}

\draw[line width=2pt, line cap=rectengle, dash pattern=on 1pt off 1]  plot [smooth,tension=1.5] coordinates{(e0)(0,2.5)(e1)};
\draw[line width=2pt, line cap=rectengle, dash pattern=on 1pt off 1]  plot [smooth,tension=1.5] coordinates{(e6)(0,-2.5)(e7)};
\draw[edge]  plot [smooth,tension=1.5] coordinates{(e0)(0,2.5)(e1)};
\draw[edge]  plot [smooth,tension=1.5] coordinates{(e6)(0,-2.5)(e7)};
\end{scope}

\foreach \x/\y in {a2/b1,a3/b0,a5/b2,a4/b3,a7/b4,a6/b5}{
\draw [edge, dotted]  (\x) -- (\y);}
\draw[edge, dotted]  plot [smooth,tension=0.1] coordinates{(a0)(-2,1)(-2,-3.5)(3,-3.5)(b7)};
\draw[edge, dotted]  plot [smooth,tension=0.1] coordinates{(a1)(1,3.5)(7,3.5)(7,-2)(b6)};

\foreach \x/\y in {b2/c1,b3/c0,b5/c2,b4/c3,b7/c4,b6/c5}{
\draw [edge]  (\x) -- (\y);
\draw [line width=2pt, line cap=rectengle, dash pattern=on 1pt off 1]  (\x) -- (\y);
}
\draw[edge]  plot [smooth,tension=0.1] coordinates{(b0)(3,1)(3,-3.5)(8,-3.5)(c7)};
\draw[edge]  plot [smooth,tension=0.1] coordinates{(b1)(6,3.5)(12,3.5)(12,-2)(c6)};
\draw[line width=2pt, line cap=rectengle, dash pattern=on 1pt off 1]  plot [smooth,tension=0.1] coordinates{(b0)(3,1)(3,-3.5)(8,-3.5)(c7)};
\draw[line width=2pt, line cap=rectengle, dash pattern=on 1pt off 1]  plot [smooth,tension=0.1] coordinates{(b1)(6,3.5)(12,3.5)(12,-2)(c6)};

\foreach \x/\y in {c2/d1,c3/d0,c5/d2,c4/d3,c7/d4,c6/d5}{
\draw [edge]  (\x) -- (\y);
\draw [line width=3pt, line cap=round, dash pattern=on 0pt off 2\pgflinewidth]  (\x) -- (\y);
}
\draw[edge]  plot [smooth,tension=0.1] coordinates{(c0)(8,1)(8,-3.5)(13,-3.5)(d7)};
\draw[edge]  plot [smooth,tension=0.1] coordinates{(c1)(11,3.5)(17,3.5)(17,-2)(d6)};
\draw[line width=3pt, line cap=round, dash pattern=on 0pt off 2\pgflinewidth]  plot [smooth,tension=0.1] coordinates{(c0)(8,1)(8,-3.5)(13,-3.5)(d7)};
\draw[line width=3pt, line cap=round, dash pattern=on 0pt off 2\pgflinewidth]  plot [smooth,tension=0.1] coordinates{(c1)(11,3.5)(17,3.5)(17,-2)(d6)};

\foreach \x/\y in {d2/e1,d3/e0,d5/e2,d4/e3,d7/e4,d6/e5}{
\draw [edge]  (\x) -- (\y);}
\draw[edge]  plot [smooth,tension=0.1] coordinates{(d0)(13,1)(13,-3.5)(18,-3.5)(e7)};
\draw[edge]  plot [smooth,tension=0.1] coordinates{(d1)(16,3.5)(22,3.5)(22,-2)(e6)};

\begin{scope}[shift={(7,0.7)}]
\node[ver] (308) at (-4,-5){$0$};
\node[ver] (300) at (-1,-5){$1$};
\node[ver] (301) at (2,-5){$2$};
\node[ver] (302) at (5,-5){$3$};
\node[ver] (303) at (8,-5){$4$};
\node[ver] (309) at (-2,-5){};
\node[ver] (304) at (1,-5){};
\node[ver] (305) at (4,-5){};
\node[ver] (306) at (7,-5){};
\node[ver] (307) at (10,-5){};
\path[edge] (300) -- (304);
\path[edge] (308) -- (309);
\draw [line width=2pt, line cap=rectengle, dash pattern=on 1pt off 1]  (308) -- (309);
\draw [line width=3pt, line cap=round, dash pattern=on 0pt off 2\pgflinewidth]  (300) -- (304);
\path[edge] (301) -- (305);
\path[edge, dashed] (302) -- (306);
\path[edge, dotted] (303) -- (307);
\end{scope}
\end{tikzpicture}
\caption{The $5$-colored graph $(\Gamma',\gamma')$ with boundary corresponding to the crystallization $(\Gamma,\gamma)$ and the  isomorphism $I(\Gamma):\Gamma \to \Gamma$ defined as in Figure \ref{fig:1}.}\label{fig:2}
\end{figure}

\noindent Then $(\Gamma',\gamma')$  is a $(d+2)$-colored graph with boundary, which is regular with respect to the color $d$. The boundary graph $(\partial \Gamma',\partial \gamma')$ has two isomorphic components with the color set $\Delta_{d+1}\setminus\{d\}$, and each component is a crystallization of $M$. Let $K_1$ and $K_2$ be the components of $(\partial \Gamma',\partial \gamma')$ with the vertex sets $\{x^{(0)}_0,x^{(1)}_{0},\dots, x^{(2p-1)}_{0}\}$ and $\{x^{(0)}_{d+1},x^{(1)}_{d+1},\dots, x^{(2p-1)}_{d+1}\}$ respectively. Then, $|\mathcal{K}(K_1)|\cong |\mathcal{K}(K_2)| \cong M$.

\medskip

\noindent {\em Claim:} we claim that the $(d+2)$-colored graph $(\Gamma',\gamma')$ with boundary represents $M\times [0,1]$.

\smallskip

Let ${\mathcal K}(\Gamma')$ be the $(d+1)$-dimensional simplicial cell-complex corresponding to the  $(d+2)$-colored graph $(\Gamma',\gamma')$ with boundary. It is easy to see that
${\mathcal K}(\Gamma')$ has exactly $2d+3$ vertices - there is a unique vertex labeled by the color $d$ but for each color $i\in \Delta_{d+1}\setminus \{d\}$, there are exactly two vertices labeled by the color $i$. Recall that $\sigma(v)$ denotes the $(d+1)$-simplex in ${\mathcal K}(\Gamma')$ corresponding to the vertex $v\in V(\Gamma')$. Let $\sigma(v)[\hat{i}]$ denote the $d$-face of $\sigma(v)$ opposite to the vertex labeled by the color $i$.

Observe that, for $0\leq j \leq 2p-1$ the $d$-faces $\sigma(x^{(j)}_0)[\hat{d}]$ of $\sigma(x^{(j)}_0)$ can continuously move to the $d$-faces $\sigma(x^{(j)}_0)[\hat{(d+1)}]$ of $\sigma(x^{(j)}_0)$ through the $(d+1)$-simplices $\sigma(x^{(j)}_0)$ by moving the first copy of the vertices labeled by $d+1$ to the unique vertex labeled by $d$ and keeping the other vertices fixed.

Further, for $0\leq j \leq 2p-1$ the $d$-faces $\sigma(I_V(x^{(j)}_0))[\hat{(d+1)}]$ of $\sigma(I_V(x^{(j)}_0))$ are identified to the $d$-faces $\sigma(x^{(j)}_1)[\hat{(d+1)}]$ of $\sigma(x^{(j)}_1)$.
Then, the $d$-faces $\sigma(x^{(j)}_1)[\hat{(d+1)}]$ of $\sigma(x^{(j)}_1)$ can continuously move to the $d$-faces $\sigma(x^{(j)}_1)[\hat{0}]$ of $\sigma(x^{(j)}_1)$ through the $(d+1)$-simplices $\sigma(x^{(j)}_1)$ by moving  the first copy of the vertices labeled by $0$ to the second copy of the vertices labeled by $d+1$ and keeping the other vertices fixed.

For each $2\leq k \leq d$, the $d$-faces $\sigma(I_V(x^{(j)}_{k-1}))[\hat{(k-2)}]$ of $\sigma(I_V(x^{(j)}_{k-1}))$ are identified to the $d$-faces $\sigma(x^{(j)}_k)[\hat{(k-2)}]$ of $\sigma(x^{(j)}_k)$. Then, the $d$-faces $\sigma(x^{(j)}_k)[\hat{(k-2)}]$ of $\sigma(x^{(j)}_k)$ can continuously move to the $d$-faces $\sigma(x^{(j)}_k)[\hat{k-1}]$ of $\sigma(x^{(j)}_k)$ through the $(d+1)$-simplices $\sigma(x^{(j)}_k)$ by moving  the first copy of the vertices labeled by $k-1$ to the second copy of the vertices labeled by $k-2$  and keeping the other vertices fixed.

Finally, the $d$-faces $\sigma(I_V(x^{(j)}_{d}))[\hat{(d-1)}]$ of the $(d+1)$-simplex $\sigma(I_V(x^{(j)}_{d}))$ are identified to the $d$-faces $\sigma(x^{(j)}_{d+1})[\hat{(d-1)}]$ of the $(d+1)$-simplex $\sigma(x^{(j)}_{d+1})$. Then, the $d$-faces $\sigma(x^{(j)}_{d+1})[\hat{(d-1)}]$ of $\sigma(x^{(j)}_{d+1})$ can continuously move to the $d$-faces $\sigma(x^{(j)}_{d+1})[\hat{d}]$ of $\sigma(x^{(j)}_{d+1})$ through the $(d+1)$-simplices $\sigma(x^{(j)}_{d+1})$ by moving  the unique vertex labeled by $d$ to the second copy of the vertices labeled by $d-1$  and keeping the other vertices fixed.

Thus, there is a continuous path (say, $T$) from $\mathcal{K}(K_1)$ to $\mathcal{K}(K_{2})$ such that, in each level $|T|$ represents $M$. This implies that $(\Gamma',\gamma')$ represents $M \times [0,1]$. For an example, in Figure \ref{Fig:S2XI}, we have considered the 2-vertex crystallization of $S^2$, and by using this construction we have constructed the 4-regular colored graph $(\Gamma',\gamma')$ with boundary. The continuous movement of simplices in $\mathcal{K}(\Gamma')$ has also been shown in the figure.

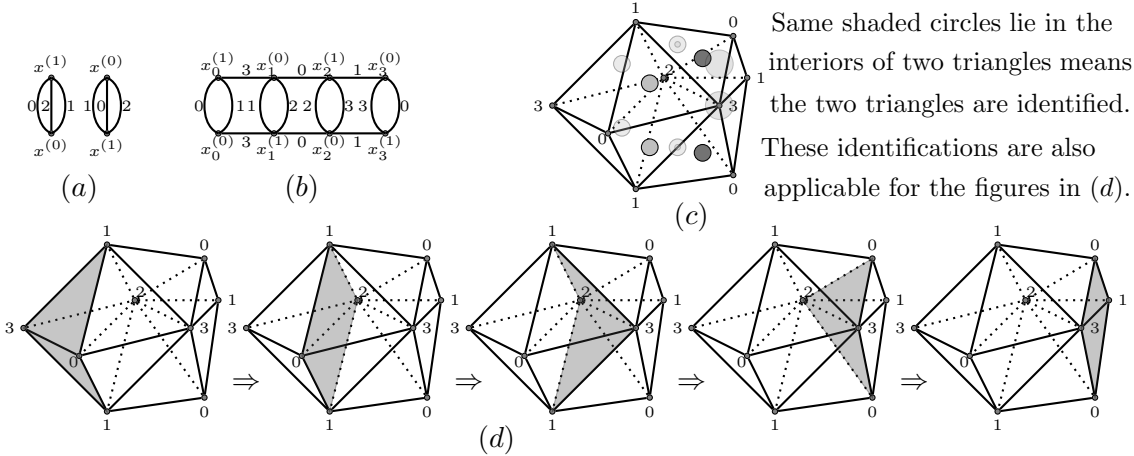
\begin{figure}[ht]
\tikzstyle{vert}=[circle, draw, fill=black!100, inner sep=0pt, minimum width=4pt]
\tikzstyle{vertex}=[circle, draw, fill=black!00, inner sep=0pt, minimum width=4pt]
\tikzstyle{ver}=[]
\tikzstyle{extra}=[circle, draw, fill=black!50, inner sep=0pt, minimum width=2pt]
\tikzstyle{edge} = [draw,thick,-]
\centering
\begin{tikzpicture}[scale=0.37]

\begin{scope}[shift={(-9,2)}]
\foreach \x/\y/\z in {-9/5/v,-9/7/v',-7/5/w,-7/7/w',-3/5/w1,-3/7/w2,-1/5/w3,-1/7/w4,1/5/w5,1/7/w6,3/5/w7,3/7/w8}{
\node[extra] (\z) at (\x,\y){};}

\foreach \x/\y/\z in {-7.7/6/1,-6.3/6/2,-7.2/6/0,-9.7/6/0,-8.3/6/1,-9.2/6/2,-3.7/6/0,-2.2/6/1,-1.8/6/1,-.3/6/2,-2/7.3/3,-2/4.7/3,.2/6/2,1.7/6/3,0/7.3/0,0/4.7/0,2.2/6/3,3.7/6/0,2/7.3/1,2/4.7/1}{
\node[ver] () at (\x,\y){\tiny{$\z$}};}
\foreach \x/\y/\z in {-9/4.5/x^{(0)},-9/7.5/x^{(1)},-7/4.5/x^{(1)},-7/7.5/x^{(0)},-3/4.5/x_0^{(0)},-3/7.5/x_0^{(1)},-1/4.5/x_1^{(1)},-1/7.5/x_1^{(0)},1/4.5/x_2^{(0)},1/7.5/x_2^{(1)},3/4.5/x_3^{(1)},3/7.5/x_3^{(0)}}{\node[ver] () at (\x,\y){\tiny{$\z$}};}

\foreach \x/\y in {v/v',w/w',w1/w3,w2/w4,w3/w5,w4/w6,w5/w7,w6/w8}{
\draw [edge]  (\x) -- (\y);}

\draw[edge]  plot [smooth,tension=1.5] coordinates{(w1)(-3.5,6)(w2)};
\draw[edge]  plot [smooth,tension=1.5] coordinates{(w1)(-2.5,6)(w2)};
\draw[edge]  plot [smooth,tension=1.5] coordinates{(w3)(-1.5,6)(w4)};
\draw[edge]  plot [smooth,tension=1.5] coordinates{(w3)(-.5,6)(w4)};
\draw[edge]  plot [smooth,tension=1.5] coordinates{(w5)(.5,6)(w6)};
\draw[edge]  plot [smooth,tension=1.5] coordinates{(w5)(1.5,6)(w6)};
\draw[edge]  plot [smooth,tension=1.5] coordinates{(w7)(2.5,6)(w8)};
\draw[edge]  plot [smooth,tension=1.5] coordinates{(w7)(3.5,6)(w8)};
\draw[edge]  plot [smooth,tension=1.5] coordinates{(w)(-7.5,6)(w')};
\draw[edge]  plot [smooth,tension=1.5] coordinates{(w)(-6.5,6)(w')};
\draw[edge]  plot [smooth,tension=1.5] coordinates{(v)(-9.5,6)(v')};
\draw[edge]  plot [smooth,tension=1.5] coordinates{(v)(-8.5,6)(v')};

\end{scope}

\begin{scope}[shift={(3,8)}]

\foreach \x/\y/\z in {-3/0/v1,1/1/v2,3/0/v3,-1/-1/v4,0/3/v5,0/-3/v6,3.5/2.5/v7,3.5/-2.5/v8,4/1/v9}{
\node[extra] (\z) at (\x,\y){};}

\foreach \x/\y in {v3/v4,v4/v1,v1/v5,v3/v5,v4/v5,v6/v1,v3/v6,v4/v6,v5/v7,v3/v7,v3/v8,v6/v8,v7/v9,v3/v9,v8/v9}{
\draw [edge]  (\x) -- (\y);}

\foreach \x/\y in {v1/v2,v2/v3,v2/v4,v2/v5,v2/v6,v2/v7,v2/v8,v2/v9}{
\draw [edge, dotted]  (\x) -- (\y);}
\foreach \x/\y/\z in {-3.5/0/3,1.2/1.3/2,3.5/0/3,-1.2/-1.2/0,0/3.5/1,0/-3.5/1,3.5/3/0,3.5/-3/0,4.5/1/1}{
\node[ver] () at (\x,\y){\tiny{$\z$}};}

\filldraw[fill=black!25,opacity=0.3] (-.5,1.5) circle (.3);
\filldraw[fill=black!25,opacity=0.3] (-.5,-.8) circle (.3);
\filldraw[fill=black!25,opacity=1] (.5,-1.5) circle (.3);
\filldraw[fill=black!25,opacity=1] (.5,.8) circle (.3);
\filldraw[fill=black!25,opacity=0.3] (1.5,2.2) circle (.1);
\filldraw[fill=black!25,opacity=0.3] (1.5,2.2) circle (.3);
\filldraw[fill=black!25,opacity=0.3] (1.5,-1.5) circle (.1);
\filldraw[fill=black!25,opacity=0.3] (1.5,-1.5) circle (.3);
\filldraw[fill=black!55,opacity=1] (2.4,1.7) circle (.3);
\filldraw[fill=black!55,opacity=1] (2.4,-1.7) circle (.3);
\filldraw[fill=black!55,opacity=0.2] (3,0) circle (.5);
\filldraw[fill=black!55,opacity=0.2] (3,1.5) circle (.5);

\node[ver] () at (11,3){\small{Same shaded circles lie in the}};
\node[ver] () at (11.3,1.5){\small{interiors of two triangles means}};
\node[ver] () at (11.3,0){\small{the two triangles are identified.}};
\node[ver] () at (10.5,-1.5){\small{These identifications are also }};
\node[ver] () at (11.2,-3){\small{applicable for the figures in $(d)$.}};
\end{scope}

\begin{scope}[shift={(-16,0)}]
\foreach \x/\y\z in {-3/0/v1,1/1/v2,3/0/v3,-1/-1/v4,0/3/v5,0/-3/v6,3.5/2.5/v7,3.5/-2.5/v8,4/1/v9}{
\node[extra] (\z) at (\x,\y){};}

\foreach \x/\y\z in {-3.5/0/3,1.2/1.3/2,3.5/0/3,-1.2/-1.2/0,0/3.5/1,0/-3.5/1,3.5/3/0,3.5/-3/0,4.5/1/1}{
\node[ver] () at (\x,\y){\tiny{$\z$}};}

\foreach \x/\y in {v3/v4,v4/v1,v1/v5,v3/v5,v4/v5,v6/v1,v3/v6,v4/v6,v5/v7,v3/v7,v3/v8,v6/v8,v7/v9,v3/v9,v8/v9}{
\draw [edge]  (\x) -- (\y);}

\foreach \x/\y in {v1/v2,v2/v3,v2/v4,v2/v5,v2/v6,v2/v7,v2/v8,v2/v9}{
\draw [edge, dotted]  (\x) -- (\y);}
\node[ver] () at (5,-2){$\Rightarrow$};
\draw[fill=black!75,opacity=0.3]  (-1,-1) -- (0,3)-- (-3,0)  -- (0,-3) -- cycle;
\end{scope}

\begin{scope}[shift={(-8,0)}]
\foreach \x/\y\z in {-3/0/v1,1/1/v2,3/0/v3,-1/-1/v4,0/3/v5,0/-3/v6,3.5/2.5/v7,3.5/-2.5/v8,4/1/v9}{
\node[extra] (\z) at (\x,\y){};}

\foreach \x/\y\z in {-3.5/0/3,1.2/1.3/2,3.5/0/3,-1.2/-1.2/0,0/3.5/1,0/-3.5/1,3.5/3/0,3.5/-3/0,4.5/1/1}{
\node[ver] () at (\x,\y){\tiny{$\z$}};}

\foreach \x/\y in {v3/v4,v4/v1,v1/v5,v3/v5,v4/v5,v6/v1,v3/v6,v4/v6,v5/v7,v3/v7,v3/v8,v6/v8,v7/v9,v3/v9,v8/v9}{
\draw [edge]  (\x) -- (\y);}

\foreach \x/\y in {v1/v2,v2/v3,v2/v4,v2/v5,v2/v6,v2/v7,v2/v8,v2/v9}{
\draw [edge, dotted]  (\x) -- (\y);}
\node[ver] () at (5,-2){$\Rightarrow$};
\draw[fill=black!75,opacity=0.3]  (-1,-1) -- (0,3)-- (1,1)  -- (0,-3) -- cycle;
\end{scope}

\begin{scope}[shift={(0,0)}]
\foreach \x/\y\z in {-3/0/v1,1/1/v2,3/0/v3,-1/-1/v4,0/3/v5,0/-3/v6,3.5/2.5/v7,3.5/-2.5/v8,4/1/v9}{
\node[extra] (\z) at (\x,\y){};}

\foreach \x/\y\z in {-3.5/0/3,1.2/1.3/2,3.5/0/3,-1.2/-1.2/0,0/3.5/1,0/-3.5/1,3.5/3/0,3.5/-3/0,4.5/1/1}{
\node[ver] () at (\x,\y){\tiny{$\z$}};}

\foreach \x/\y in {v3/v4,v4/v1,v1/v5,v3/v5,v4/v5,v6/v1,v3/v6,v4/v6,v5/v7,v3/v7,v3/v8,v6/v8,v7/v9,v3/v9,v8/v9}{
\draw [edge]  (\x) -- (\y);}

\foreach \x/\y in {v1/v2,v2/v3,v2/v4,v2/v5,v2/v6,v2/v7,v2/v8,v2/v9}{
\draw [edge, dotted]  (\x) -- (\y);}
\node[ver] () at (5,-2){$\Rightarrow$};
\draw[fill=black!75,opacity=0.3]  (3,0) -- (0,3)-- (1,1)  -- (0,-3) -- cycle;
\end{scope}

\begin{scope}[shift={(8,0)}]
\foreach \x/\y\z in {-3/0/v1,1/1/v2,3/0/v3,-1/-1/v4,0/3/v5,0/-3/v6,3.5/2.5/v7,3.5/-2.5/v8,4/1/v9}{
\node[extra] (\z) at (\x,\y){};}

\foreach \x/\y\z in {-3.5/0/3,1.2/1.3/2,3.5/0/3,-1.2/-1.2/0,0/3.5/1,0/-3.5/1,3.5/3/0,3.5/-3/0,4.5/1/1}{
\node[ver] () at (\x,\y){\tiny{$\z$}};}

\foreach \x/\y in {v3/v4,v4/v1,v1/v5,v3/v5,v4/v5,v6/v1,v3/v6,v4/v6,v5/v7,v3/v7,v3/v8,v6/v8,v7/v9,v3/v9,v8/v9}{
\draw [edge]  (\x) -- (\y);}

\foreach \x/\y in {v1/v2,v2/v3,v2/v4,v2/v5,v2/v6,v2/v7,v2/v8,v2/v9}{
\draw [edge, dotted]  (\x) -- (\y);}
\node[ver] () at (5,-2){$\Rightarrow$};
\draw[fill=black!75,opacity=0.3]  (3,0) -- (3.5,2.5)-- (1,1)  -- (3.5,-2.5) -- cycle;
\end{scope}

\begin{scope}[shift={(16,0)}]
\foreach \x/\y\z in {-3/0/v1,1/1/v2,3/0/v3,-1/-1/v4,0/3/v5,0/-3/v6,3.5/2.5/v7,3.5/-2.5/v8,4/1/v9}{
\node[extra] (\z) at (\x,\y){};}

\foreach \x/\y\z in {-3.5/0/3,1.2/1.3/2,3.5/0/3,-1.2/-1.2/0,0/3.5/1,0/-3.5/1,3.5/3/0,3.5/-3/0,4.5/1/1}{
\node[ver] () at (\x,\y){\tiny{$\z$}};}

\foreach \x/\y in {v3/v4,v4/v1,v1/v5,v3/v5,v4/v5,v6/v1,v3/v6,v4/v6,v5/v7,v3/v7,v3/v8,v6/v8,v7/v9,v3/v9,v8/v9}{
\draw [edge]  (\x) -- (\y);}

\foreach \x/\y in {v1/v2,v2/v3,v2/v4,v2/v5,v2/v6,v2/v7,v2/v8,v2/v9}{
\draw [edge, dotted]  (\x) -- (\y);}
\draw[fill=black!75,opacity=0.3]  (3,0) -- (3.5,2.5)-- (4,1)  -- (3.5,-2.5) -- cycle;
\end{scope}

\node[ver] () at (-17,5){$(a)$};
\node[ver] () at (-9,5){$(b)$};
\node[ver] () at (5,4){$(c)$};
\node[ver] () at (-2,-4){$(d)$};

\end{tikzpicture}
\caption{ $(a)$ A crystallization $(\Gamma,\gamma)$ of $\mathbb{S}^2$ and an isomorphism $I(\Gamma):\Gamma \to \Gamma$ such that $I_V(x^{(p)})=x^{(p+1)}$ (addition is modulo $2$) and $I_c(i)=i+1$ (addition is modulo $3$), $(b)$ The corresponding $4$-colored graph $(\Gamma',\gamma')$ with boundary, $(c)$
$\mathcal{K}(\Gamma')$, $(d)$ Topological realization of the continuous movement of simplices from $\mathbb{S}^2\times \{0\}$ to $\mathbb{S}^2\times\{1\}$ in $\mathbb{S}^2\times [0,1] \cong |\mathcal{K}(\Gamma')|$.}\label{Fig:S2XI}
\end{figure}

{\em Construction of $\bar \Gamma$:} Note that the continuous path $T:\mathcal{K}(K_1)\to \mathcal{K}(K_{2})$ sends the first copy of vertices labeled by the colors
$d+1,0,1,\dots,d-1$ to the second copy of vertices labeled by the colors  $d-1,d+1,0,\dots,d-2$. Here, $|\mathcal{K}(K_1)|\cong |\mathcal{K}(K_2)| \cong M$.

Now, let  $\tilde I(\Gamma):=(\tilde I_V,\tilde I_c):\Gamma \to \Gamma$ be an isomorphism ($\tilde I(\Gamma)$ may be different from $I(\Gamma)$) with the property $\tilde I_c(i)=i+1$ (addition is modulo $d+1$) for $i \in \Delta_d$. Then, we have $\tilde I(K_{2}):\mathcal{K}(K_{2}) \to \mathcal{K}(K_{2})$ such that $\tilde I(K_{2})(d-1)=d+1$, $\tilde I(K_{2})(d+1)=0$ and $\tilde I(K_{2})(i)=i+1$ for $i\in\{0,1,\dots,d-2\}$ and $|\tilde I(K_{2})|:|\mathcal{K}(K_{2})| \to |\mathcal{K}(K_{2})|$ is a (PL) homeomorphism. Thus, $|\tilde I(K_{2})|\circ |T|:M\to M$ is a (PL) homeomorphism such that $\tilde I(K_{2})\circ T:\mathcal{K}(K_1)\to \mathcal{K}(K_{2})$ sends each vertex to the same colored copy of itself. Therefore, we can construct a  crystallization $(\bar \Gamma,\bar\gamma)$ of mapping torus of the map $f:=|\tilde I(K_{2})|\circ |T|:M\to M$ by joining  the vertices  $x^{(j)}_0$ and $\tilde I_V(x^{(j)}_{d+1})$ by a $d$-colored edge  for each $0\leq j \leq 2p-1$.  Then, by the construction $\bar \Gamma$ has $(d+2)2p=2(d+2)p$ vertices. \end{proof}

\begin{example}{\rm
Let $(\Gamma,\gamma)$ be the unique 8-vertex crystallization of $\mathbb{S}^2 \times \mathbb{S}^1$ with color set $\Delta_3$ (cf. \cite{bd14}). Let $x^{(0)},x^{(1)},\dots, x^{(7)}$ be the vertices of $\Gamma$. Let $I(\Gamma):\Gamma \to \Gamma$ be an isomorphism such that $I_V(x^{(p)})=x^{(p+2)}$ (addition is modulo $8$) for all $p\in\{0,1,\dots,7\}$ and $I_c(i)=i+1$ (addition is modulo $4$) for all $i\in \Delta_3$. Then by the construction given in Lemma \ref{lemma:1}, we can construct $\Gamma^k$ for $0 \leq k \leq 4$ (cf. Figure \ref{fig:1}). Further,  by the construction given in Lemma \ref{lemma:1}, we can construct
the $5$-colored graph $(\Gamma',\gamma')$ with boundary  as in Figure \ref{fig:2}. Again, by the construction in Lemma \ref{lemma:1}, $(\Gamma',\gamma')$ represents $\mathbb{S}^2 \times \mathbb{S}^1 \times [0,1]$. Now, if we choose $\tilde{I}(\Gamma)=I(\Gamma)$ as in the proof of Lemma \ref{lemma:1}, then we get a crystallization $(\bar\Gamma,\bar\gamma)$ of a mapping torus of a (PL) homeomorphism $f:\mathbb{S}^2 \times \mathbb{S}^1\to \mathbb{S}^2 \times \mathbb{S}^1$ with 40 vertices (cf. Figure \ref{fig:S2S1}). Later, in the proof of Theorem \ref{theorem:genus1}, we shall see that the regular genus of the mapping torus is 6.

\begin{figure}[ht]
\tikzstyle{vert}=[circle, draw, fill=black!100, inner sep=0pt, minimum width=4pt]
\tikzstyle{vertex}=[circle, draw, fill=black!00, inner sep=0pt, minimum width=4pt]
\tikzstyle{ver}=[]
\tikzstyle{extra}=[circle, draw, fill=black!50, inner sep=0pt, minimum width=2pt]
\tikzstyle{edge} = [draw,thick,-]

\centering
\begin{tikzpicture}[scale=0.515]
\begin{scope}[shift={(0,0)}]
\foreach \w/\x/\y/\z in {-1.5/-1/3/0,1.5/1/3/1,1.5/1/1/2,-1.5/-1/1/3,-1.5/-1/-1/4,1.5/1/-1/5,1.5/1/-3/6,-1.5/-1/-3/7}{
\node[vert] (a\z) at (\x,\y){};}

\node[ver] () at (-1.5,3.2){\tiny{$ x_{0}^{(2)}$}};
\node[ver] () at (1.7,2.8){\tiny{$ x_{0}^{(3)}$}};
\node[ver] () at (-1.5,1.2){\tiny{$ x_{0}^{(4)}$}};
\node[ver] () at (1.7,.8){\tiny{$ x_{0}^{(5)}$}};
\node[ver] () at (-1.5,-.8){\tiny{$ x_{0}^{(6)}$}};
\node[ver] () at (1.7,-1.2){\tiny{$ x_{0}^{(7)}$}};
\node[ver] () at (-1.5,-2.8){\tiny{$ x_{0}^{(0)}$}};
\node[ver] () at (1.7,-3.2){\tiny{$ x_{0}^{(1)}$}};

\foreach \x/\y in {a2/a1,a0/a3,a7/a6,a5/a4}{
\draw [line width=2pt, line cap=rectengle, dash pattern=on 1pt off 1]  (\x) -- (\y);
\draw [edge]  (\x) -- (\y);}

\foreach \x/\y in {a0/a1,a2/a3,a5/a6,a7/a4,a3/a4,a2/a5}{
\draw [edge]  (\x) -- (\y);}

\foreach \x/\y in {a3/a4,a2/a5}{
\draw [line width=3pt, line cap=round, dash pattern=on 0pt off 2\pgflinewidth]  (\x) -- (\y);}

\draw[line width=3pt, line cap=round, dash pattern=on 0pt off 2\pgflinewidth]  plot [smooth,tension=1.5] coordinates{(a0)(0,2.5)(a1)};
\draw[line width=3pt, line cap=round, dash pattern=on 0pt off 2\pgflinewidth]  plot [smooth,tension=1.5] coordinates{(a6)(0,-2.5)(a7)};
\draw[edge]  plot [smooth,tension=1.5] coordinates{(a0)(0,2.5)(a1)};
\draw[edge]  plot [smooth,tension=1.5] coordinates{(a6)(0,-2.5)(a7)};

\end{scope}

\begin{scope}[shift={(5,0)}]
\foreach \w/\x/\y/\z in {-1.5/-1/3/0,1.5/1/3/1,1.5/1/1/2,-1.5/-1/1/3,-1.5/-1/-1/4,1.5/1/-1/5,1.5/1/-3/6,-1.5/-1/-3/7}{
\node[vert] (b\z) at (\x,\y){};}

\node[ver] () at (-1.5,3.2){\tiny{$ x_{1}^{(2)}$}};
\node[ver] () at (1.7,2.8){\tiny{$ x_{1}^{(3)}$}};
\node[ver] () at (-1.5,1.2){\tiny{$ x_{1}^{(4)}$}};
\node[ver] () at (1.7,.8){\tiny{$ x_{1}^{(5)}$}};
\node[ver] () at (-1.5,-.8){\tiny{$ x_{1}^{(6)}$}};
\node[ver] () at (1.7,-1.2){\tiny{$ x_{1}^{(7)}$}};
\node[ver] () at (-1.5,-2.8){\tiny{$ x_{1}^{(0)}$}};
\node[ver] () at (1.7,-3.2){\tiny{$ x_{1}^{(1)}$}};

\foreach \x/\y in {b2/b1,b0/b3,b7/b6,b5/b4}{
\draw [line width=3pt, line cap=round, dash pattern=on 0pt off 2\pgflinewidth]  (\x) -- (\y);
\draw [edge]  (\x) -- (\y);}

\foreach \x/\y in {b3/b4,b2/b5}{
\draw [edge]  (\x) -- (\y);}

\foreach \x/\y in {b0/b1,b2/b3,b5/b6,b7/b4}{
\draw [edge, dashed]  (\x) -- (\y);}

\draw[edge]  plot [smooth,tension=1.5] coordinates{(b0)(0,2.5)(b1)};
\draw[edge]  plot [smooth,tension=1.5] coordinates{(b6)(0,-2.5)(b7)};
\end{scope}

\begin{scope}[shift={(10,0)}]
\foreach \w/\x/\y/\z in {-1.5/-1/3/0,1.5/1/3/1,1.5/1/1/2,-1.5/-1/1/3,-1.5/-1/-1/4,1.5/1/-1/5,1.5/1/-3/6,-1.5/-1/-3/7}{
\node[vert] (c\z) at (\x,\y){};}
\node[ver] () at (-1.5,3.2){\tiny{$ x_{2}^{(2)}$}};
\node[ver] () at (1.7,2.8){\tiny{$ x_{2}^{(3)}$}};
\node[ver] () at (-1.5,1.2){\tiny{$ x_{2}^{(4)}$}};
\node[ver] () at (1.7,.8){\tiny{$ x_{2}^{(5)}$}};
\node[ver] () at (-1.5,-.8){\tiny{$ x_{2}^{(6)}$}};
\node[ver] () at (1.7,-1.2){\tiny{$ x_{2}^{(7)}$}};
\node[ver] () at (-1.5,-2.8){\tiny{$ x_{2}^{(0)}$}};
\node[ver] () at (1.7,-3.2){\tiny{$ x_{2}^{(1)}$}};
\foreach \x/\y in {c0/c3,c1/c2,c5/c4,c6/c7}{
\draw [edge]  (\x) -- (\y);}

\foreach \x/\y in {c0/c1,c2/c3,c5/c6,c7/c4}{
\draw [edge, dotted]  (\x) -- (\y);}

\foreach \x/\y in {c3/c4,c2/c5}{
\draw [edge, dashed]  (\x) -- (\y);}

\draw[edge, dashed] plot [smooth,tension=1.5] coordinates{(c0)(0,2.5)(c1)};
\draw[edge, dashed]  plot [smooth,tension=1.5] coordinates{(c6)(0,-2.5)(c7)};
\end{scope}

\begin{scope}[shift={(15,0)}]
\foreach \w/\x/\y/\z in {-1.5/-1/3/0,1.5/1/3/1,1.5/1/1/2,-1.5/-1/1/3,-1.5/-1/-1/4,1.5/1/-1/5,1.5/1/-3/6,-1.5/-1/-3/7}{
\node[vert] (d\z) at (\x,\y){};}
\node[ver] () at (-1.5,3.2){\tiny{$ x_{3}^{(2)}$}};
\node[ver] () at (1.7,2.8){\tiny{$ x_{3}^{(3)}$}};
\node[ver] () at (-1.5,1.2){\tiny{$ x_{3}^{(4)}$}};
\node[ver] () at (1.7,.8){\tiny{$ x_{3}^{(5)}$}};
\node[ver] () at (-1.5,-.8){\tiny{$ x_{3}^{(6)}$}};
\node[ver] () at (1.7,-1.2){\tiny{$ x_{3}^{(7)}$}};
\node[ver] () at (-1.5,-2.8){\tiny{$ x_{3}^{(0)}$}};
\node[ver] () at (1.7,-3.2){\tiny{$ x_{3}^{(1)}$}};
\foreach \x/\y in {d2/d1,d0/d3,d7/d6,d5/d4}{\draw [edge, dashed]  (\x) -- (\y);}

\foreach \x/\y in {d0/d1,d2/d3,d5/d6,d7/d4}{
\draw [edge]  (\x) -- (\y);
\draw [line width=2pt, line cap=rectengle, dash pattern=on 1pt off 1]  (\x) -- (\y);
}

\foreach \x/\y in {d3/d4,d2/d5}{
\draw [edge, dotted]  (\x) -- (\y);}

\draw[edge, dotted]  plot [smooth,tension=1.5] coordinates{(d0)(0,2.5)(d1)};
\draw[edge, dotted]  plot [smooth,tension=1.5] coordinates{(d6)(0,-2.5)(d7)};
\end{scope}

\begin{scope}[shift={(20,0)}]
\foreach \w/\x/\y/\z in {-1.5/-1/3/0,1.5/1/3/1,1.5/1/1/2,-1.5/-1/1/3,-1.5/-1/-1/4,1.5/1/-1/5,1.5/1/-3/6,-1.5/-1/-3/7}{
\node[vert] (e\z) at (\x,\y){};}
\node[ver] () at (-1.5,3.2){\tiny{$ x_{4}^{(2)}$}};
\node[ver] () at (1.7,2.8){\tiny{$ x_{4}^{(3)}$}};
\node[ver] () at (-1.5,1.2){\tiny{$ x_{4}^{(4)}$}};
\node[ver] () at (1.7,.8){\tiny{$ x_{4}^{(5)}$}};
\node[ver] () at (-1.5,-.8){\tiny{$ x_{4}^{(6)}$}};
\node[ver] () at (1.7,-1.2){\tiny{$ x_{4}^{(7)}$}};
\node[ver] () at (-1.5,-2.8){\tiny{$ x_{4}^{(0)}$}};
\node[ver] () at (1.7,-3.2){\tiny{$ x_{4}^{(1)}$}};
\foreach \x/\y in {e2/e1,e0/e3,e7/e6,e5/e4}{\draw [edge, dotted]  (\x) -- (\y);}
\foreach \x/\y in {e0/e1,e2/e3,e5/e6,e7/e4}{
\draw [edge]  (\x) -- (\y);
\draw [line width=3pt, line cap=round, dash pattern=on 0pt off 2\pgflinewidth]  (\x) -- (\y);
}

\foreach \x/\y in {e3/e4,e2/e5}{
\draw [edge]  (\x) -- (\y);}

\foreach \x/\y in {e3/e4,e2/e5}{
\draw [line width=2pt, line cap=rectengle, dash pattern=on 1pt off 1]  (\x) -- (\y);
}

\draw[line width=2pt, line cap=rectengle, dash pattern=on 1pt off 1]  plot [smooth,tension=1.5] coordinates{(e0)(0,2.5)(e1)};
\draw[line width=2pt, line cap=rectengle, dash pattern=on 1pt off 1]  plot [smooth,tension=1.5] coordinates{(e6)(0,-2.5)(e7)};
\draw[edge]  plot [smooth,tension=1.5] coordinates{(e0)(0,2.5)(e1)};
\draw[edge]  plot [smooth,tension=1.5] coordinates{(e6)(0,-2.5)(e7)};
\end{scope}

\begin{scope}[shift={(25,0)}]
\foreach \w/\x/\y/\z in {-1.5/-1/3/0,1.5/1/3/1,1.5/1/1/2,-1.5/-1/1/3,-1.5/-1/-1/4,1.5/1/-1/5,1.5/1/-3/6,-1.5/-1/-3/7}{
\node[vert] (f\z) at (\x,\y){};}
\node[ver] () at (-1.5,3.2){\tiny{$ x_{0}^{(2)}$}};
\node[ver] () at (1.7,2.8){\tiny{$ x_{0}^{(3)}$}};
\node[ver] () at (-1.5,1.2){\tiny{$ x_{0}^{(4)}$}};
\node[ver] () at (1.7,.8){\tiny{$ x_{0}^{(5)}$}};
\node[ver] () at (-1.5,-.8){\tiny{$ x_{0}^{(6)}$}};
\node[ver] () at (1.7,-1.2){\tiny{$ x_{0}^{(7)}$}};
\node[ver] () at (-1.5,-2.8){\tiny{$ x_{0}^{(0)}$}};
\node[ver] () at (1.7,-3.2){\tiny{$ x_{0}^{(1)}$}};
\foreach \x/\y in {f2/f1,f0/f3,f7/f6,f5/f4}{
\draw [line width=2pt, line cap=rectengle, dash pattern=on 1pt off 1]  (\x) -- (\y);
\draw [edge]  (\x) -- (\y);}

\foreach \x/\y in {f0/f1,f2/f3,f5/f6,f7/f4,f3/f4,f2/f5}{
\draw [edge]  (\x) -- (\y);}

\foreach \x/\y in {f3/f4,f2/f5}{
\draw [line width=3pt, line cap=round, dash pattern=on 0pt off 2\pgflinewidth]  (\x) -- (\y);}

\draw[line width=3pt, line cap=round, dash pattern=on 0pt off 2\pgflinewidth]  plot [smooth,tension=1.5] coordinates{(f0)(0,2.5)(f1)};
\draw[line width=3pt, line cap=round, dash pattern=on 0pt off 2\pgflinewidth]  plot [smooth,tension=1.5] coordinates{(f6)(0,-2.5)(f7)};
\draw[edge]  plot [smooth,tension=1.5] coordinates{(f0)(0,2.5)(f1)};
\draw[edge]  plot [smooth,tension=1.5] coordinates{(f6)(0,-2.5)(f7)};
\end{scope}

\foreach \x/\y in {a2/b1,a3/b0,a5/b2,a4/b3,a7/b4,a6/b5}{
\draw [edge, dotted]  (\x) -- (\y);}
\draw[edge, dotted]  plot [smooth,tension=0.1] coordinates{(a0)(-2,1)(-2,-3.5)(3,-3.5)(b7)};
\draw[edge, dotted]  plot [smooth,tension=0.1] coordinates{(a1)(1,3.5)(7,3.5)(7,-2)(b6)};

\foreach \x/\y in {b2/c1,b3/c0,b5/c2,b4/c3,b7/c4,b6/c5}{
\draw [edge]  (\x) -- (\y);
\draw [line width=2pt, line cap=rectengle, dash pattern=on 1pt off 1]  (\x) -- (\y);
}
\draw[edge]  plot [smooth,tension=0.1] coordinates{(b0)(3,1)(3,-3.5)(8,-3.5)(c7)};
\draw[edge]  plot [smooth,tension=0.1] coordinates{(b1)(6,3.5)(12,3.5)(12,-2)(c6)};
\draw[line width=2pt, line cap=rectengle, dash pattern=on 1pt off 1]  plot [smooth,tension=0.1] coordinates{(b0)(3,1)(3,-3.5)(8,-3.5)(c7)};
\draw[line width=2pt, line cap=rectengle, dash pattern=on 1pt off 1]  plot [smooth,tension=0.1] coordinates{(b1)(6,3.5)(12,3.5)(12,-2)(c6)};

\foreach \x/\y in {c2/d1,c3/d0,c5/d2,c4/d3,c7/d4,c6/d5}{
\draw [edge]  (\x) -- (\y);
\draw [line width=3pt, line cap=round, dash pattern=on 0pt off 2\pgflinewidth]  (\x) -- (\y);
}
\draw[edge]  plot [smooth,tension=0.1] coordinates{(c0)(8,1)(8,-3.5)(13,-3.5)(d7)};
\draw[edge]  plot [smooth,tension=0.1] coordinates{(c1)(11,3.5)(17,3.5)(17,-2)(d6)};
\draw[line width=3pt, line cap=round, dash pattern=on 0pt off 2\pgflinewidth]  plot [smooth,tension=0.1] coordinates{(c0)(8,1)(8,-3.5)(13,-3.5)(d7)};
\draw[line width=3pt, line cap=round, dash pattern=on 0pt off 2\pgflinewidth]  plot [smooth,tension=0.1] coordinates{(c1)(11,3.5)(17,3.5)(17,-2)(d6)};

\foreach \x/\y in {d2/e1,d3/e0,d5/e2,d4/e3,d7/e4,d6/e5}{
\draw [edge]  (\x) -- (\y);}
\draw[edge]  plot [smooth,tension=0.1] coordinates{(d0)(13,1)(13,-3.5)(18,-3.5)(e7)};
\draw[edge]  plot [smooth,tension=0.1] coordinates{(d1)(16,3.5)(22,3.5)(22,-2)(e6)};

\foreach \x/\y in {e2/f1,e3/f0,e5/f2,e4/f3,e7/f4,e6/f5}{
\draw [edge, dashed]  (\x) -- (\y);}
\draw[edge, dashed]  plot [smooth,tension=0.1] coordinates{(e0)(18,1)(18,-3.5)(23,-3.5)(f7)};
\draw[edge, dashed]  plot [smooth,tension=0.1] coordinates{(e1)(21,3.5)(27,3.5)(27,-2)(f6)};

 \begin{scope}[shift={(5,0)}]
\node[ver] (308) at (-4,-5){$0$};
\node[ver] (300) at (-1,-5){$1$};
\node[ver] (301) at (2,-5){$2$};
\node[ver] (302) at (5,-5){$3$};
\node[ver] (303) at (8,-5){$4$};
\node[ver] (309) at (-2,-5){};
\node[ver] (304) at (1,-5){};
\node[ver] (305) at (4,-5){};
\node[ver] (306) at (7,-5){};
\node[ver] (307) at (10,-5){};
\path[edge] (300) -- (304);
\path[edge] (308) -- (309);
\draw [line width=2pt, line cap=rectengle, dash pattern=on 1pt off 1]  (308) -- (309);
\draw [line width=3pt, line cap=round, dash pattern=on 0pt off 2\pgflinewidth]  (300) -- (304);
\path[edge] (301) -- (305);
\path[edge, dashed] (302) -- (306);
\path[edge, dotted] (303) -- (307);
\end{scope}
\end{tikzpicture}
\caption{A 40-vertex crystallization $(\bar\Gamma,\bar \gamma)$ of a mapping torus of a (PL) homeomorphism $f:\mathbb{S}^2\times \mathbb{S}^1 \to \mathbb{S}^2\times \mathbb{S}^1$ with regular genus 6 (the vertices and edges of the two components of $\bar\Gamma_{\{0,1,2\}}$ with vertices $x_0^{(0)},x_0^{(1)}\dots,x_0^{(7)}$ must be identified in the above figure).}\label{fig:S2S1}
\end{figure}
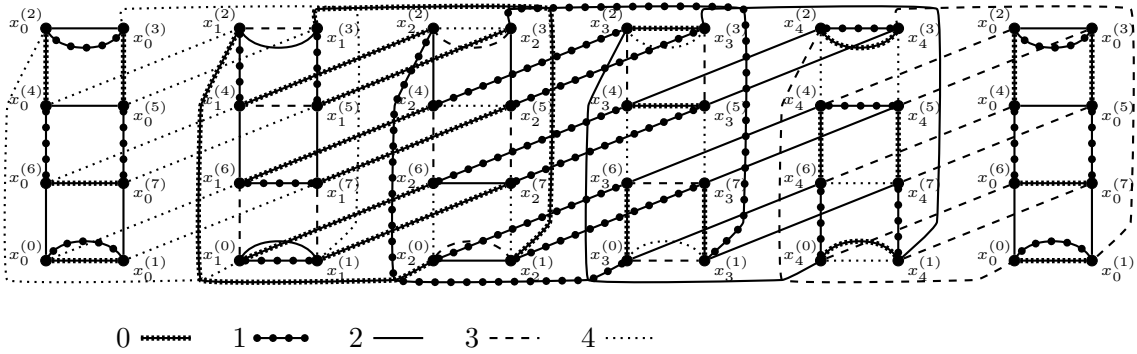
}
\end{example}

If $\Gamma$ is non-bipartite (i.e., $M$ is non-orientable) then  $\bar \Gamma$ is also non-bipartite (i.e., $M_f$ is also non-orientable). For $d \geq 1$, let $(\Gamma,\gamma)$ be a $2p$-vertex crystallization of a closed connected  $d$-manifold $M$. Let the vertices of the crystallization $(\Gamma,\gamma)$ be $x^{(0)},x^{(1)}$, $\dots, x^{(2p-1)}$. If $M$ is orientable then $(\Gamma,\gamma)$ is a bipartite graph. Let $\{x^{(2u)},0\leq u \leq p-1\}$ and $\{x^{(2u-1)},1\leq u \leq p\}$ be the partition of the vertex set.

\begin{lemma} \label{lemma:2}
For $d \geq 1$, let $(\Gamma,\gamma)$ be a $2p$-vertex crystallization (with color set $\Delta_d$) of a closed connected orientable $d$-manifold $M$. If, there exist two isomorphisms $I'(\Gamma):=(I'_V,I'_c),I''(\Gamma):=(I''_V,I''_c):\Gamma \to \Gamma$ such that $I'_V(x^{(2u')})=x^{(2v')}$, $I''_V(x^{(2u''-1)})=x^{(2v'')}$ for some $u',v',u''v''$, and $I'_c(i)=I''_c(i)=i+1$ (addition is modulo $d+1$) for $i\in \Delta_d$, then there exists a crystallization $(\bar \Gamma,\bar \gamma)$ of an orientable (resp.,  non-orientable) mapping torus $M_f$ (resp., $M_{\tilde f}$) of a (PL) homeomorphism $f:M\to M$ (resp., $\tilde f:M\to M$) with $2(d+2)p$ vertices.

\end{lemma}

\begin{proof}
Let $\tilde I(\Gamma), I(\Gamma)$, $T$ and $(\bar \Gamma,\bar\gamma)$ be as in the proof of Lemma \ref{lemma:1}. Since $M$ is orientable, $\Gamma$ is bipartite.

\begin{itemize}
\item If $\tilde I(\Gamma)=I(\Gamma)=I''(\Gamma)$ then $\bar\Gamma$ is bipartite. Thus $(\bar \Gamma,\bar\gamma)$ represents an orientable mapping torus $M_f$ of the map $f:=|I''(\Gamma)|\circ |T|:M\to M$ with $2(d+2)p$ vertices.

\item If $d$ is odd and $\tilde I(\Gamma)=I(\Gamma)=I'(\Gamma)$ then we have an odd cycle in $\bar \Gamma$. Thus $(\bar \Gamma,\bar\gamma)$ represents a non-orientable mapping torus $M_{\tilde f}$ of the map $\tilde f:=|I'(\Gamma)|\circ |T|:M\to M$ with $2(d+2)p$ vertices.

\item If $d$ is even and $\tilde I(\Gamma)\neq I(\Gamma) \in \{I'(\Gamma),I''(\Gamma)\}$ then then we have an odd cycle in $\bar \Gamma$. Thus $(\bar \Gamma,\bar\gamma)$ represents a non-orientable mapping torus $M_{\tilde f}$ of the map $\tilde f:=|\tilde I(\Gamma)|\circ |T|:M\to M$ with $2(d+2)p$ vertices.
\end{itemize}
\end{proof}

\begin{remark}\label{remark:component}
{\rm For $0 \leq j \leq d+1$,  $\bar \Gamma_{\{j,j+1\}}$ (addition is modulo $d+2$) is of type $pC_6 \sqcup (d-1)\Gamma_{\{0,1\}}$ and hence $\bar g_{\{j,j+1\}}=p+(d-1)g_{\{0,1\}}$. If $j$ and $k$ are not consecutive colors then $\bar \Gamma_{\{j,k\}}$  is of type $2pC_4 \sqcup \Gamma_{\{j_1,k_1\}}\sqcup\dots \sqcup\Gamma_{\{j_{d-2},k_{d-2}\}}$ for $d\geq 2$, where  $j_i$ and $k_i$ are not consecutive colors. 

In particular, if $d=3$ then $g_{\{0,2\}}=g_{\{1,3\}}$ (cf. \cite{fgg86}). Thus, by the construction, if $j$ and $k$ are not consecutive colors then $\bar \Gamma_{\{j,k\}}$  is of type $2pC_4 \sqcup \Gamma_{\{j_1,k_1\}}$, i.e., $\bar g_{\{j,k\}}=2p+g_{\{0,2\}}$. Further,
from \cite{fgg86}, we know that $g_{\{0,1\}} + g_{\{0,2\}}+g_{\{0,3\}}=2+p$. On the other hand, the existence of an isomorphism that cyclically permutes the colors of the edges implies $g_{\{0,1\}} = g_{\{1,2\}}=g_{\{2,3\}}=g_{\{0,3\}}$ (supposing the permutation to be the identity). Therefore, $2g_{\{0,1\}} + g_{\{0,2\}}=2+p$. Thus, the existence of an isomorphism of a crystallization of a $3$-manifold that cyclically permutes the colors of the edges implies that $g_{\{0,2\}}$ (=$g_{\{1,3\}}$) and $p$ must have the same parity (supposing the permutation to be the identity). }
\end{remark}

\begin{remark}\label{remark:isomorphism}
{\rm
Let $(\Gamma,\gamma)$, $(\bar \Gamma,\bar \gamma)$, $I(\Gamma)$ and $\tilde I(\Gamma)$ be as in the proof of Lemma \ref{lemma:2}. Let $V(\Gamma)=\{x^{(0)}$, $x^{(1)}, \dots, x^{(2p-1)}\}$ and $\bar \Gamma$ be the crystallization of  $M_f$ constructed as in the proof of Lemma \ref{lemma:2} with an isomorphism $I(\Gamma)=\tilde I(\Gamma)$. Therefore, $\{x_0^{(0)},\dots, x_0^{(2p-1)},x_1^{(0)},\dots, x_1^{(2p-1)}$, $\dots,x_{d+1}^{(0)}, \dots, x_{d+1}^{(2p-1)}\}$ is the vertex set of $\bar \Gamma$. From the construction, it is clear that there exists an isomorphism $\bar I(\bar \Gamma):=(\bar I_V,\bar I_c): \bar \Gamma \to \bar \Gamma$ with $\bar I_c(i)=i+1$ (addition is module $d+2$) for $0\leq i \leq d+1$  and $\bar I_V(x_j^{(s)})=x_{j+1}^{(s)}$  (addition in the subscript is module $d+2$) for all $s\in\{0,\dots,2p-1\}$. Thus, we can use Lemma \ref{lemma:2} again and, we can construct a crystallization  $\bar {\bar {\Gamma}}$  of a mapping torus of a (PL) homeomorphism $g:M_f\to M_f$.
}
\end{remark}

\begin{lemma} \label{lemma:3}
Let $M_f$ be the mapping torus of $f:M\to M$ constructed as in Lemmas \ref{lemma:1} and \ref{lemma:2}. Then $\mathit{k}(M_f) \leq (d+2)p-1$. Moreover, if $d=3$ then $\mathcal{G}(M_f) \leq 1+5(p-g_{\{0,2\}})/2$.
\end{lemma}

\begin{proof}
Let  $(\bar \Gamma,\bar \gamma)$ be the crystallization of the mapping torus $M_f$ constructed as in Lemmas \ref{lemma:1} and \ref{lemma:2}. Then, by the construction $\bar \Gamma$ has $(d+2)2p=2(d+2)p$ vertices, i.e., $\mathit{k}(M_f) \leq (d+2)p-1$.

If $d=3$ then $\bar{\Gamma}_{\{2j,2j+2\}}$ (here the addition is modulo $5$) is of the form $\Gamma_{\{0,2\}} \sqcup 2p C_4$, i.e.,  $\bar{g}_{\{2j,2j+2\}}=g_{\{0,2\}}+2p$ for $0 \leq j \leq 4$. Now, for the cyclic permutation $\varepsilon=(0,2,4,1,3)$, we have $\chi_{\varepsilon}(\bar \Gamma)= \sum_{i \in \mathbb{Z}_{5}}\bar g_{\{\varepsilon_i,\varepsilon_{i+1}\}}-3  \frac{\#V(\bar \Gamma)}{2}=5(g_{\{0,2\}}+2p)-3(5p)=5g_{\{0,2\}}-5p$. Thus, $\rho_{\varepsilon}(\bar \Gamma) = 1 - \chi_{\varepsilon}(\bar\Gamma)/2 = 1 + 5(p-g_{\{0,2\}})/2$ and hence $\mathcal{G}(M_f) \leq 1 + 5(p-g_{\{0,2\}})/2$.
\end{proof}

\begin{remark}\label{remark:sphere}
{\rm Observe that if $M \cong \mathbb{S}^d$ and $(\Gamma,\gamma)$ is the standard 2-vertex crystallization of $\mathbb{S}^d$ then by choosing two possible isomorphims as in Lemma \ref{lemma:2}, $\bar \Gamma$ gives the $2(d+2)$-vertex crystallization  of $\mathbb{S}^d \times \mathbb{S}^1$ and $\mathbb{S}^{\hspace{.2mm}d} \mbox{$\times\hspace{-2.6mm}_{-}$} \, \mathbb{S}^{\hspace{.1mm}1}$ constructed as in Figures \ref{fig:SdS1:1} and \ref{fig:SdS1:2} (which is already known from \cite{gg87}). From \cite{[FG$_2$]}, we know $\mathcal G(\mathbb{S}^d \times \mathbb{S}^1)=\mathcal G(\mathbb{S}^{\hspace{.2mm}d} \mbox{$\times\hspace{-2.6mm}_{-}$} \, \mathbb{S}^{\hspace{.1mm}1})=1$.
}
\end{remark}

\begin{example}\label{example:surface}
{\rm
For $1\leq h \leq 2$, let $(\Gamma,\gamma)$ be a crystallization of $U_h$. Let $I(\Gamma):=(I_V,I_c):\Gamma \to \Gamma$ be an isomorphism such that $I_c(i)=i+1$ (addition is modulo 3) for $i\in \Delta_2$, and $I_V$ sends the vertices $x^{(0)},x^{(1)},x^{(2)},x^{(3)}$ to $x^{(1)},x^{(2)},x^{(0)},x^{(3)}$ respectively for $h=1$ (cf. Figure \ref{fig:surface} (a)) and sends the vertices $x^{(0)},x^{(1)},x^{(2)},x^{(3)},x^{(4)},x^{(5)}$ to $x^{(5)},x^{(0)},x^{(3)},x^{(4)},x^{(2)},x^{(1)}$ respectively for $h=2$ (cf. Figure \ref{fig:surface} (b)). 
In both cases, if we choose $\tilde I(\Gamma)= I''(\Gamma) =I(\Gamma)$ as in Lemma \ref{lemma:2}, then there is a crystallization $(\bar \Gamma,\bar \gamma)$ of a mapping torus $(U_h)_f$ of a (PL) homeomorphism $f:U_h \to U_h$. By the construction, $\bar \Gamma$ has 16 (resp., 24) vertices for $h=1$ (resp., $h=2$). Therefore, $\mathit{k}((U_1)_f) \leq 7$ and $\mathit{k}((U_2)_f) \leq 11$. Now, by Proposition \ref{prop:gagliardi79b}, it is easy to prove that $\pi_1((U_h)_f)=\pi_1(U_h)\times\mathbb{Z}$. Therefore, from table III of \cite{ca98} (also available in \cite{bcg09}), we have $(U_1)_f=U_1\times \mathbb{S}^1$, $(U_2)_f=U_2\times \mathbb{S}^1$, $\mathit{k} (U_1\times \mathbb{S}^1)=7$ and $\mathit{k} (U_2\times \mathbb{S}^1)=11$. Further, from \cite{fg82} we know $\mathcal{G} (U_1\times \mathbb{S}^1)=2$ (resp., $\mathcal{G} (U_2\times \mathbb{S}^1)=3$) which is actually $\rho(\bar \Gamma)$.

Let $(\Gamma,\gamma)$ be a crystallization of $T_1$. Let $I(\Gamma):=(I_V,I_c):\Gamma \to \Gamma$ be an isomorphism such that $I_c(i)=i+1$ (addition is modulo 3) for $i\in \Delta_2$, and $I_V$ sends the vertices $x^{(0)},x^{(1)},x^{(2)},x^{(3)},x^{(4)},x^{(5)}$ to $x^{(3)},x^{(4)},x^{(1)},x^{(2)},x^{(5)},x^{(0)}$ respectively (cf. Figure \ref{fig:surface} (c)).
If we choose $\tilde I(\Gamma)= I''(\Gamma) =I(\Gamma)$ as in Lemma \ref{lemma:2}, then there is a crystallization $(\bar \Gamma,\bar \gamma)$ of a mapping torus $(T_1)_f$ of a (PL) homeomorphism $f:T_1 \to T_1$. By the construction, $\bar \Gamma$ has $24$ vertices. Observe that $(\bar \Gamma,\bar \gamma)$ is actually the crystallization
obtained in \cite[Figure 5]{bd14}. Thus, $(\bar \Gamma,\bar \gamma)$  is the unique minimal crystallization of $T_1 \times \mathbb{S}^1$.  Further, from \cite{fg82} we know $\mathcal{G} (T_1\times \mathbb{S}^1)=3$ which is actually $\rho(\bar \Gamma)$.
}\end{example}

\begin{figure}[ht]
\tikzstyle{vert}=[circle, draw, fill=black!100, inner sep=0pt, minimum width=4pt]
\tikzstyle{vertex}=[circle, draw, fill=black!00, inner sep=0pt, minimum width=4pt]
\tikzstyle{ver}=[]
\tikzstyle{extra}=[circle, draw, fill=black!50, inner sep=0pt, minimum width=2pt]
\tikzstyle{edge} = [draw,thick,-]
\centering
\begin{tikzpicture}[scale=.4]
\begin{scope}[shift={(-13,0)}]
\foreach \x/\y/\z in {60/$x^{(0)}$/a0,240/$x^{(2)}$/a2}{
\node[ver] () at (\x:4){\tiny{\y}};
\node[vert] (\z) at (\x:3){};
} 

\foreach \x/\y/\z in
{120/$x^{(1)}$/a1,300/$x^{(3)}$/a3}{ 
\node[ver] () at (\x:4){\tiny{\y}};
\node[vert] (\z) at (\x:3){};
} 

\foreach \x/\y in{a0/a1,a1/a2,a2/a3,a3/a0,a0/a2,a1/a3}{\path[edge] (\x) -- (\y);}

\foreach \x/\y in{a0/a1,a2/a3}{\draw [line width=2pt, line cap=rectengle, dash pattern=on 1pt off 1]  (\x) -- (\y);}

\foreach \x/\y in{a1/a2,a3/a0}{\draw[line width=3pt, line cap=round, dash pattern=on 0pt off 2\pgflinewidth] (\x) -- (\y);}
\end{scope}

\begin{scope}[shift={(-13,-10)}]
\foreach \x/\y/\z in {60/$x^{(1)}$/a0,240/$x^{(0)}$/a2}{
\node[ver] () at (\x:4){\tiny{\y}};
\node[vert] (\z) at (\x:3){};
} 

\foreach \x/\y/\z in
{120/$x^{(2)}$/a1,300/$x^{(3)}$/a3}{ 
\node[ver] () at (\x:4){\tiny{\y}};
\node[vert] (\z) at (\x:3){};
} 

\foreach \x/\y in{a0/a1,a1/a2,a2/a3,a3/a0,a0/a2,a1/a3}{\path[edge] (\x) -- (\y);}

\foreach \x/\y in{a0/a2,a1/a3}{\draw [line width=2pt, line cap=rectengle, dash pattern=on 1pt off 1]  (\x) -- (\y);}

\foreach \x/\y in{a0/a1,a2/a3}{\draw[line width=3pt, line cap=round, dash pattern=on 0pt off 2\pgflinewidth] (\x) -- (\y);}

\node[ver] () at (0,-5.4){\tiny{$(a)$ A crystallization}};
\node[ver] () at (0,-6){\tiny{of $U_1$}};
\end{scope}

\begin{scope}[shift={(-5,0)}]
\foreach \x/\y/\z in {0/$x^{(0)}$/a0,120/$x^{(2)}$/a2,240/$x^{(4)}$/a4}{
\node[ver] () at (\x:4){\tiny{\y}};
\node[vert] (\z) at (\x:3){};
} 

\foreach \x/\y/\z in
{60/$x^{(1)}$/a1,180/$x^{(3)}$/a3,300/$x^{(5)}$/a5}{ 
\node[ver] () at (\x:4){\tiny{\y}};
\node[vert] (\z) at (\x:3){};
} 

\foreach \x/\y in{a0/a1,a1/a2,a2/a3,a3/a4,a4/a5,a5/a0,a0/a3,a1/a5,a2/a4}{\path[edge] (\x) -- (\y);}

\foreach \x/\y in{a0/a1,a2/a3,a4/a5}{\draw [line width=2pt, line cap=rectengle, dash pattern=on 1pt off 1]  (\x) -- (\y);}

\foreach \x/\y in{a1/a2,a3/a4,a5/a0}{\draw[line width=3pt, line cap=round, dash pattern=on 0pt off 2\pgflinewidth] (\x) -- (\y);}
\end{scope}

\begin{scope}[shift={(-5,-10)}]
\foreach \x/\y/\z in {0/$x^{(5)}$/a0,120/$x^{(3)}$/a2,240/$x^{(2)}$/a4}{
\node[ver] () at (\x:4){\tiny{\y}};
\node[vert] (\z) at (\x:3){};
} 

\foreach \x/\y/\z in
{60/$x^{(0)}$/a1,180/$x^{(4)}$/a3,300/$x^{(1)}$/a5}{ 
\node[ver] () at (\x:4){\tiny{\y}};
\node[vert] (\z) at (\x:3){};
} 

\foreach \x/\y in{a0/a1,a1/a2,a2/a3,a3/a4,a4/a5,a5/a0,a0/a3,a1/a5,a2/a4}{\path[edge] (\x) -- (\y);}

\foreach \x/\y in{a0/a1,a2/a3,a4/a5}{\draw[line width=3pt, line cap=round, dash pattern=on 0pt off 2\pgflinewidth] (\x) -- (\y);}

\foreach \x/\y in{a0/a3,a1/a5,a2/a4}{\draw [line width=2pt, line cap=rectengle, dash pattern=on 1pt off 1]  (\x) -- (\y);}

\node[ver] () at (0,-5.4){\tiny{$(b)$ A crystallization}};
\node[ver] () at (0,-6){\tiny{of $U_2$}};
\end{scope}

\begin{scope}[shift={(5,0)}]
\foreach \x/\y/\z in {0/$x^{(0)}$/a0,120/$x^{(2)}$/a2,240/$x^{(4)}$/a4}{
\node[ver] () at (\x:4){\tiny{\y}};
\node[vertex] (\z) at (\x:3){};
} 

\foreach \x/\y/\z in
{60/$x^{(1)}$/a1,180/$x^{(3)}$/a3,300/$x^{(5)}$/a5}{ 
\node[ver] () at (\x:4){\tiny{\y}};
\node[vert] (\z) at (\x:3){};
} 

\foreach \x/\y in{a0/a1,a1/a2,a2/a3,a3/a4,a4/a5,a5/a0,a0/a3,a1/a4,a2/a5}{\path[edge] (\x) -- (\y);}

\foreach \x/\y in{a0/a1,a2/a3,a4/a5}{\draw [line width=2pt, line cap=rectengle, dash pattern=on 1pt off 1]  (\x) -- (\y);}

\foreach \x/\y in{a1/a2,a3/a4,a5/a0}{\draw[line width=3pt, line cap=round, dash pattern=on 0pt off 2\pgflinewidth] (\x) -- (\y);}
\end{scope}

\begin{scope}[shift={(5,-10)}]
\foreach \x/\y/\z in {0/$x^{(3)}$/a0,120/$x^{(1)}$/a2,240/$x^{(5)}$/a4}{
\node[ver] () at (\x:4){\tiny{\y}};
\node[vert] (\z) at (\x:3){};
} 

\foreach \x/\y/\z in
{60/$x^{(4)}$/a1,180/$x^{(2)}$/a3,300/$x^{(0)}$/a5}{ 
\node[ver] () at (\x:4){\tiny{\y}};
\node[vertex] (\z) at (\x:3){};
} 

\foreach \x/\y in{a0/a1,a1/a2,a2/a3,a3/a4,a4/a5,a5/a0,a0/a3,a1/a4,a2/a5}{\path[edge] (\x) -- (\y);}

\foreach \x/\y in{a0/a1,a2/a3,a4/a5}{
\draw[line width=3pt, line cap=round, dash pattern=on 0pt off 2\pgflinewidth] (\x) -- (\y);}

\foreach \x/\y in{a0/a3,a1/a4,a2/a5}{\draw [line width=2pt, line cap=rectengle, dash pattern=on 1pt off 1]  (\x) -- (\y);}

\node[ver] () at (0,-5.4){\tiny{$(c)$ A crystallization}};
\node[ver] () at (0,-6){\tiny{of $T_1$}};
\end{scope}

\begin{scope}[shift={(16,0)}]
\foreach \x/\y/\z in {0/$x^{(0)}$/a0,72/$x^{(2)}$/a2,144/$x^{(4)}$/a4,216/$x^{(6)}$/a6,288/$x^{(8)}$/a8}{
\node[ver] () at (\x:4.5){\tiny{\y}};
\node[vertex] (\z) at (\x:3.5){};
} 

\foreach \x/\y/\z in
{36/$x^{(1)}$/a1,108/$x^{(3)}$/a3,180/$x^{(5)}$/a5,252/$x^{(7)}$/a7,324/$x^{(9)}$/a9}{ 
\node[ver] () at (\x:4.5){\tiny{\y}};
\node[vert] (\z) at (\x:3.5){};
} 

\foreach \x/\y in{a0/a1,a1/a2,a2/a3,a3/a4,a4/a5,a5/a6,a6/a7,a7/a8,a8/a9,a9/a0,a0/a5,a1/a8,a9/a2,a6/a3,a4/a7}{\path[edge] (\x) -- (\y);}

\foreach \x/\y in{a0/a1,a2/a3,a4/a5,a6/a7,a8/a9}{\draw [line width=2pt, line cap=rectengle, dash pattern=on 1pt off 1]  (\x) -- (\y);}

\foreach \x/\y in{a1/a2,a3/a4,a5/a6,a7/a8,a9/a0}{\draw[line width=3pt, line cap=round, dash pattern=on 0pt off 2\pgflinewidth] (\x) -- (\y);}
\end{scope}

\begin{scope}[shift={(16,-10)}]
\foreach \x/\y/\z in {0/$x^{(3)}$/a0,72/$x^{(7)}$/a2,144/$x^{(1)}$/a4,216/$x^{(9)}$/a6,288/$x^{(5)}$/a8}{
\node[ver] () at (\x:4.5){\tiny{\y}};
\node[vert] (\z) at (\x:3.5){};
} 

\foreach \x/\y/\z in
{36/$x^{(4)}$/a1,108/$x^{(8)}$/a3,180/$x^{(2)}$/a5,252/$x^{(0)}$/a7,324/$x^{(6)}$/a9}{ 
\node[ver] () at (\x:4.5){\tiny{\y}};
\node[vertex] (\z) at (\x:3.5){};
} 

\foreach \x/\y in{a0/a1,a1/a2,a2/a3,a3/a4,a4/a5,a5/a6,a6/a7,a7/a8,a8/a9,a9/a0,a0/a5,a1/a8,a9/a2,a6/a3,a4/a7}{\path[edge] (\x) -- (\y);}

\foreach \x/\y in{a0/a1,a2/a3,a4/a5,a6/a7,a8/a9}{
\draw[line width=3pt, line cap=round, dash pattern=on 0pt off 2\pgflinewidth] (\x) -- (\y);}

\foreach \x/\y in{a0/a5,a1/a8,a9/a2,a6/a3,a4/a7}{\draw [line width=2pt, line cap=rectengle, dash pattern=on 1pt off 1]  (\x) -- (\y);}
\node[ver] () at (0,-5.4){\tiny{$(d)$ A crystallization of}};
\node[ver] () at (0,-6){\tiny{$T_2$}};
\end{scope}

 \begin{scope}[shift={(-10,0)}]
\node[ver] (308) at (-4,-5){$0$};
\node[ver] (300) at (2,-5){$1$};
\node[ver] (301) at (8,-5){$2$};
\node[ver] (309) at (0,-5){};
\node[ver] (304) at (6,-5){};
\node[ver] (305) at (12,-5){};
\path[edge] (300) -- (304);
\path[edge] (308) -- (309);
\draw [line width=2pt, line cap=rectengle, dash pattern=on 1pt off 1]  (308) -- (309);
\draw [line width=3pt, line cap=round, dash pattern=on 0pt off 2\pgflinewidth]  (300) -- (304);
\path[edge] (301) -- (305);
\end{scope}

\end{tikzpicture}
\caption{Crystallizations of $T_g$ and $U_h$ for $1\leq g,h\leq 2$, and their corresponding isomorphisms.}\label{fig:surface}
\end{figure}
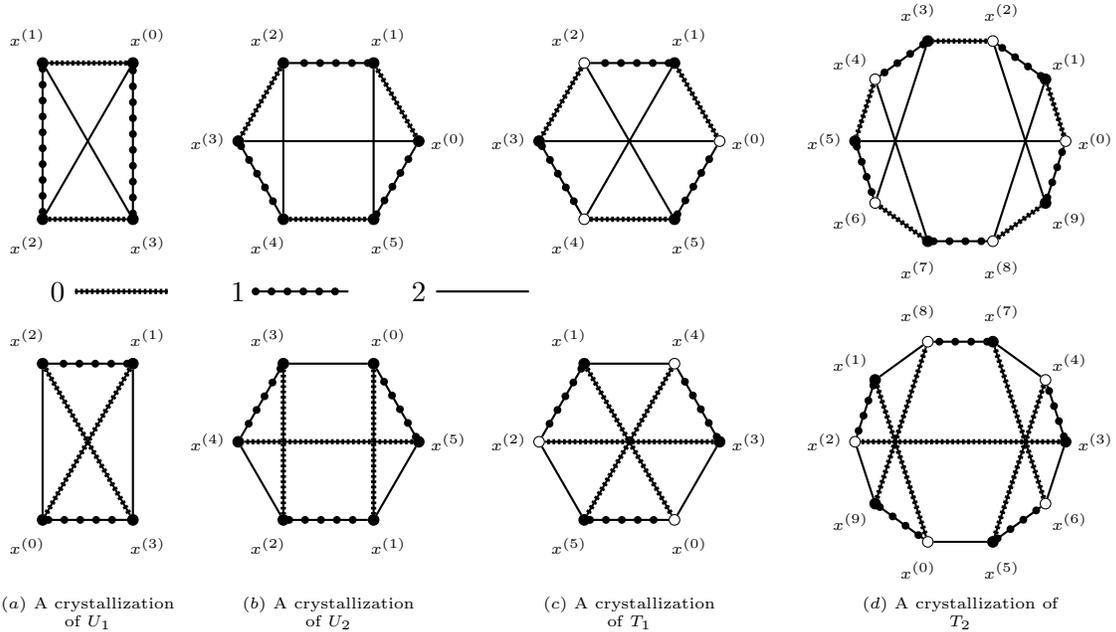

\begin{remark}
{\rm In Remark \ref{remark:sphere} and Example \ref{example:surface}, we have constructed known crystallizations of $\mathbb{S}^d \times \mathbb{S}^1$ for $d\geq 1$, $U_h \times \mathbb{S}^1$ for $1\leq h \leq 2$ and $T_1 \times \mathbb{S}^1$. In \cite{gg93}, the authors constructed colored graphs representing product of manifolds. But their construction does not give the crystallization directly, we shall have to remove the dipoles to obtain a crystallization. Moreover, if  $(\Gamma,\gamma)$ is a $2p$-vertex crystallization a PL $d$-manifold $M$ then their construction gives a  $(d+2)$-colored graph which represents $M\times S^1$, and the number of vertices is $4(d+1)p$. The advantage of our construction is that it directly gives a crystallization of a mapping torus of a (PL) homeomorphism $f:M\to M$, and the number of vertices is $2(d+2)p$. In particular, if $M=\mathbb{S}^d$ for $d\geq 1$, $U_h$ for $1\leq h \leq 2$ and $T_1$ then our construction gives a crystallization of $M\times \mathbb{S}^1$ with $2(d+2)p$, but the construction in \cite{gg93} gives a colored graph (which is not a crystallization) with $4(d+1)p$ vertices. 
}
\end{remark}

As we have pointed out in Remark \ref{remark:sphere} and Example \ref{example:surface} that the mapping torus $M_f$ is PL homeomorphic to $M\times \mathbb{S}^1$ if $M$ is real projective plane, Klein bottle, 2-torus or $d$-sphere $\mathbb{S}^d$, we believe the following is true.

\begin{conjecture}\label{conjecture:true}
If $M=\mathbb{RP}^3$ or $\mathbb{S}^{d-1}\times \mathbb{S}^1$ for $d\geq 2$, then the mapping torus $M_f$ (constructed in Theorems $\ref{theorem:genus1}$ and $\ref{theorem:genus2}$) is PL homeomorphic to $M\times \mathbb{S}^1$.
\end{conjecture}

This would follow if one can show that the (PL) homeomorphism $f:M\to M$ is (PL) isotopic to identity, i.e., there exists a PL map $H:M\times [0,1]\to M$
such that $H(*,0)=f$, $H(*,1)=Id$ (identity) and for each $t$, $H(*,t)$ is a PL homeomorphism from $M$ to itself.

\begin{figure}[ht]
\tikzstyle{vert}=[circle, draw, fill=black!100, inner sep=0pt, minimum width=4pt]
\tikzstyle{vertex}=[circle, draw, fill=black!00, inner sep=0pt, minimum width=4pt]
\tikzstyle{ver}=[]
\tikzstyle{extra}=[circle, draw, fill=black!50, inner sep=0pt, minimum width=2pt]
\tikzstyle{edge} = [draw,thick,-]
\centering
\begin{tikzpicture}[scale=.73]

\begin{scope}[shift={(0,0)}]
\foreach \x/\y/\z in {-1/5/0,1/5/1,-1/3/2,1/3/3,-1/1/4,1/1/5,-1/-1/6,1/-1/7,-1/-3/8,1/-3/9,-1/-5/10,1/-5/11}{\node[vert] (a\z) at (\x,\y){};}
\foreach \x/\y in {-1/-1.5,-1/-2,-1/-2.5,1/-1.5,1/-2,1/-2.5}{\node[extra] () at (\x,\y){};}
\foreach \x/\y in {a0/a2,a2/a4,a4/a6,a8/a10,a1/a3,a3/a5,a5/a7,a9/a11}
{\draw [edge]  (\x) -- (\y);}

\foreach \x/\y/\z in {a0/a1/4.5,a0/a1/5.5,a2/a3/2.5,a2/a3/3.5,a4/a5/.5,a4/a5/1.5,a6/a7/-.5,a6/a7/-1.5,a8/a9/-2.5,a8/a9/-3.5,a10/a11/-4.5,a10/a11/-5.5}{
\draw[edge]  plot [smooth,tension=1.5] coordinates{(\x)(0,\z)(\y)};}

\draw[edge]  plot [smooth,tension=1.2] coordinates{(a0)(-2.5,3.5)(-2.5,-3.5)(a10)};
\draw[edge]  plot [smooth,tension=1.2] coordinates{(a1)(2.5,3.5)(2.5,-3.5)(a11)};

\node[ver] () at (-1.5,5){\tiny{$ x^{(0)}$}};
\node[ver] () at (1.5,5){\tiny{$ x^{(1)}$}};
\node[ver] () at (-1.5,3){\tiny{$ x^{(2)}$}};
\node[ver] () at (1.5,3){\tiny{$ x^{(3)}$}};
\node[ver] () at (-1.5,1){\tiny{$ x^{(4)}$}};
\node[ver] () at (1.5,1){\tiny{$ x^{(5)}$}};
\node[ver] () at (-1.5,-1){\tiny{$ x^{(6)}$}};
\node[ver] () at (1.5,-1){\tiny{$ x^{(7)}$}};
\node[ver] () at (-1.7,-3){\tiny{$ x^{(2d-2)}$}};
\node[ver] () at (1.7,-3){\tiny{$ x^{(2d-1)}$}};
\node[ver] () at (-1.5,-5){\tiny{$ x^{(2d)}$}};
\node[ver] () at (1.7,-5){\tiny{$ x^{(2d+1)}$}};

\node[ver] () at (0,5.1){\tiny{$0,1,\dots,$}};
\node[ver] () at (0,4.8){\tiny{$d-2$}};
\node[ver] () at (0,3.1){\tiny{$1,2,\dots,$}};
\node[ver] () at (0,2.8){\tiny{$d-1$}};
\node[ver] () at (0,1){\tiny{$2,3,\dots,d$}};
\node[ver] () at (0,-1){\tiny{$3,\dots,d,0$}};
\node[ver] () at (0,-2.8){\tiny{$d-1,d,$}};
\node[ver] () at (0,-3.1){\tiny{$0,\dots,d-4$}};
\node[ver] () at (0,-4.9){\tiny{$d,0,1,\dots,$}};
\node[ver] () at (0,-5.2){\tiny{$d-3$}};

\foreach \x/\y/\z in {-1.3/4/$d$,1.3/4/$d$,-1.3/2/0,1.3/2/0,-1.3/0/1,1.3/0/1,-1.5/-4/$d-2$,1.5/-4/$d-2$,-2.5/0/$d-1$,2.5/0/$d-1$}{
\node[ver] () at (\x,\y){\tiny{\z}};}
\end{scope}

\begin{scope}[shift={(8,0)}]
\foreach \x/\y/\z in {-1/5/0,1/5/1,-1/3/2,1/3/3,-1/1/4,1/1/5,-1/-1/6,1/-1/7,-1/-3/8,1/-3/9,-1/-5/10,1/-5/11}{\node[vert] (a\z) at (\x,\y){};}
\foreach \x/\y in {-1/-1.5,-1/-2,-1/-2.5,1/-1.5,1/-2,1/-2.5}{\node[extra] () at (\x,\y){};}
\foreach \x/\y in {a0/a2,a2/a4,a4/a6,a8/a10,a1/a3,a3/a5,a5/a7,a9/a11}
{\draw [edge]  (\x) -- (\y);}

\foreach \x/\y/\z in {a0/a1/4.5,a0/a1/5.5,a2/a3/2.5,a2/a3/3.5,a4/a5/.5,a4/a5/1.5,a6/a7/-.5,a6/a7/-1.5,a8/a9/-2.5,a8/a9/-3.5,a10/a11/-4.5,a10/a11/-5.5}{
\draw[edge]  plot [smooth,tension=1.5] coordinates{(\x)(0,\z)(\y)};}

\draw[edge]  plot [smooth,tension=1.2] coordinates{(a0)(-2.5,3.5)(-2.5,-3.5)(a10)};
\draw[edge]  plot [smooth,tension=1.2] coordinates{(a1)(2.5,3.5)(2.5,-3.5)(a11)};

\node[ver] () at (-1.5,5){\tiny{$ x^{(2)}$}};
\node[ver] () at (1.5,5){\tiny{$ x^{(3)}$}};
\node[ver] () at (-1.5,3){\tiny{$ x^{(4)}$}};
\node[ver] () at (1.5,3){\tiny{$ x^{(5)}$}};
\node[ver] () at (-1.5,1){\tiny{$ x^{(6)}$}};
\node[ver] () at (1.5,1){\tiny{$ x^{(7)}$}};
\node[ver] () at (-1.5,-1){\tiny{$ x^{(8)}$}};
\node[ver] () at (1.5,-1){\tiny{$ x^{(9)}$}};
\node[ver] () at (-1.5,-3){\tiny{$ x^{(2d)}$}};
\node[ver] () at (1.7,-3){\tiny{$ x^{(2d+1)}$}};
\node[ver] () at (-1.5,-5){\tiny{$ x^{(0)}$}};
\node[ver] () at (1.5,-5){\tiny{$ x^{(1)}$}};

\node[ver] () at (0,5.1){\tiny{$1,2,\dots,$}};
\node[ver] () at (0,4.8){\tiny{$d-1$}};
\node[ver] () at (0,3){\tiny{$2,3,\dots,d$}};
\node[ver] () at (0,1){\tiny{$3,\dots,d,0$}};
\node[ver] () at (0,-1){\tiny{$4,\dots,d,0,1$}};
\node[ver] () at (0,-2.9){\tiny{$d,0,1,\dots,$}};
\node[ver] () at (0,-3.2){\tiny{$d-3$}};
\node[ver] () at (0,-4.9){\tiny{$0,1,\dots,$}};
\node[ver] () at (0,-5.2){\tiny{$d-2$}};

\foreach \x/\y/\z in {-1.3/4/0,1.3/4/0,-1.3/2/1,1.3/2/1,-1.3/0/2,1.3/0/2,-1.5/-4/$d-1$,1.5/-4/$d-1$,-2.5/0/$d$,2.5/0/$d$}{
\node[ver] () at (\x,\y){\tiny{\z}};}
\end{scope}
\end{tikzpicture}
\vspace{-5mm}
\caption{The standard $2(d+1)$-vertex crystallization $(\Gamma,\gamma)$ of $\mathbb{S}^{d-1}\times \mathbb{S}^1$ (resp., $\mathbb{S}^{\hspace{.2mm}d-1} \mbox{$\times\hspace{-2.8mm}_{-}$} \, \mathbb{S}^{\hspace{.1mm}1}$) if $d$ is odd (resp., even), and an isomorphism $I(\Gamma):\Gamma \to \Gamma$ such that $I_V(x^{(p)})=x^{(p+2)}$ (addition is modulo $2d+2$) for $p \in \{0,1,\dots,2d+1\}$ and $I_c(i)=i+1$ (addition is modulo $d+1$) for $i\in \Delta_d$.}\label{fig:SdS1:1}
\end{figure}
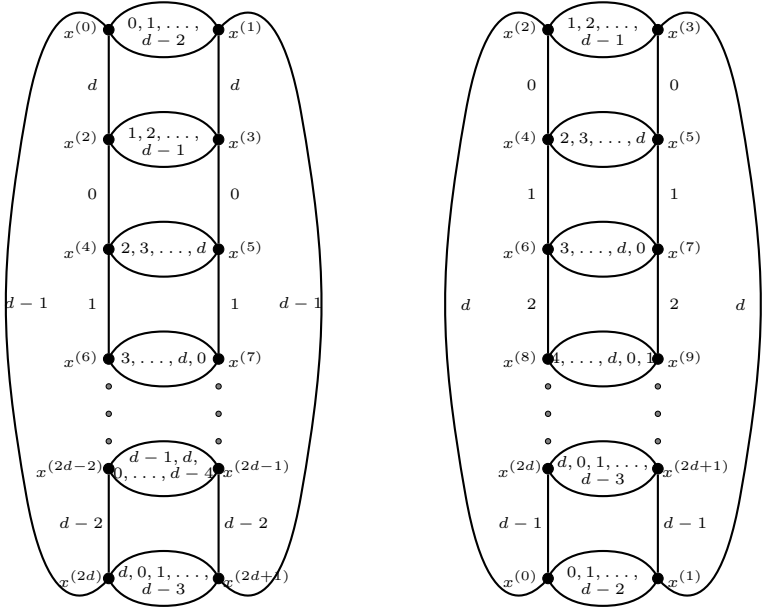

\begin{remark}{\rm
Let $(\Gamma,\gamma)$ be a crystallization of $T_2$. Let $I(\Gamma):=(I_V,I_c):\Gamma \to \Gamma$ be an isomorphism such that $I_c(i)=i+1$ (addition is modulo 3) for $i\in \Delta_2$, and $I_V$ sends the vertices $x^{(0)},x^{(1)},x^{(2)},x^{(3)},x^{(4)},x^{(5)},x^{(6)},x^{(7)},x^{(8)},x^{(9)}$ to $x^{(3)}$, $x^{(4)}$, $x^{(7)}$, $x^{(8)}$, $x^{(1)}$, $x^{(2)}$, $x^{(9)}$, $x^{(0)}$, $x^{(5)}$, $x^{(6)}$ respectively (cf. Figure \ref{fig:surface} (d)). By the symmetry of the crystallization $\Gamma$, there exist isomorphisms $I'(\Gamma),I''(\Gamma):\Gamma \to \Gamma$ as in Lemma \ref{lemma:2}. Thus, by Lemma \ref{lemma:2}, there is a crystallization $(\bar \Gamma,\bar \gamma)$ of an orientable (resp., a non-orientable) mapping torus $(T_2)_f$ (resp., $(T_2)_{\tilde f}$) of a (PL) homeomorphism $f:T_2 \to T_2$ (resp., $\tilde f:T_2 \to T_2$). For $g\geq 3$ (resp., $h\geq 3$), we do not know any crystallization $(\Gamma,\gamma)$ of $T_g$ (resp., $U_h$) with such a simple isomorphism $I(\Gamma)$ (i.e., the isomorphism which cyclically permutes the colors of the edges (supposing the permutation to be the identity)).

But, by the discussions as in Remark \ref{remark:isomorphism}, if we can construct a crystallization $(\bar \Gamma,\bar \gamma)$ of a mapping torus $M_f$ then we can apply  Lemma \ref{lemma:2} again and, we can construct a crystallization  $\bar {\bar {\Gamma}}$  of  mapping tori of some (PL) homeomorphisms $g:M_f\to M_f$. Thus, our constructions give crystallizations of mapping tori for a large class of PL $d$-manifolds.
}
\end{remark}

\medskip

\noindent {\em Proof of Theorem}  \ref{theorem:genus1}.
Let $M= \mathbb{S}^{d-1}\times \mathbb{S}^1$ or $\mathbb{S}^{\hspace{.2mm}d-1} \mbox{$\times \hspace{-2.6mm}_{-}$} \, \mathbb{S}^{\hspace{.1mm}1}$. Let $(\Gamma,\gamma)$ be the standard $2(d+1)$-vertex crystallization of $M$. Let $I(\Gamma):\Gamma \to \Gamma$ be an isomorphism  such that $I_V(x^{(p)})=x^{(p+2)}$ (addition is modulo $2d+2$) for $p \in \{0,1,\dots,2d+1\}$ and $I_c(i)=i+1$ (addition is modulo $d+1$) for $i\in \Delta_d$ (cf. Figures \ref{fig:SdS1:1} and \ref{fig:SdS1:2}). If we choose $\tilde I(\Gamma)=I''(\Gamma)=I(\Gamma)$ as in the proof of Lemma \ref{lemma:2}, then there is a crystallization $(\bar \Gamma,\bar \gamma)$ of a mapping torus $M_f$ of a (PL) homeomorphism $f:M\to M$ with $(d+2)\cdot 2(d+1)=2(d+1)(d+2)$. From the construction, $M_f$ is orientable (resp., non-orientable) if $M= \mathbb{S}^{d-1}\times \mathbb{S}^1$ (resp., $\mathbb{S}^{\hspace{.2mm}d-1} \mbox{$\times \hspace{-2.6mm}_{-}$} \, \mathbb{S}^{\hspace{.1mm}1}$).

\begin{figure}[ht]
\tikzstyle{vert}=[circle, draw, fill=black!100, inner sep=0pt, minimum width=4pt]
\tikzstyle{vertex}=[circle, draw, fill=black!00, inner sep=0pt, minimum width=4pt]
\tikzstyle{ver}=[]
\tikzstyle{extra}=[circle, draw, fill=black!50, inner sep=0pt, minimum width=2pt]
\tikzstyle{edge} = [draw,thick,-]
\centering
\begin{tikzpicture}[scale=.73]

\begin{scope}[shift={(0,0)}]
\foreach \x/\y/\z in {-1/5/0,1/5/1,-1/3/2,1/3/3,-1/1/4,1/1/5,-1/-1/6,1/-1/7,-1/-3/8,1/-3/9,-1/-5/10,1/-5/11}{\node[vert] (a\z) at (\x,\y){};}
\foreach \x/\y in {-1/-1.5,-1/-2,-1/-2.5,1/-1.5,1/-2,1/-2.5}{\node[extra] () at (\x,\y){};}
\foreach \x/\y in {a0/a2,a2/a4,a4/a6,a8/a10,a1/a3,a3/a5,a5/a7,a9/a11}
{\draw [edge]  (\x) -- (\y);}

\foreach \x/\y/\z in {a0/a1/4.5,a0/a1/5.5,a2/a3/2.5,a2/a3/3.5,a4/a5/.5,a4/a5/1.5,a6/a7/-.5,a6/a7/-1.5,a8/a9/-2.5,a8/a9/-3.5,a10/a11/-4.5,a10/a11/-5.5}{
\draw[edge]  plot [smooth,tension=1.5] coordinates{(\x)(0,\z)(\y)};}

\draw[edge]  plot [smooth,tension=1.5] coordinates{(a0)(-2.5,3.5)(-2.5,-4.5)(a11)};
\draw[edge]  plot [smooth,tension=1.5] coordinates{(a1)(2.5,3.5)(2.5,-4.5)(a10)};

\node[ver] () at (-1.5,5){\tiny{$ x^{(0)}$}};
\node[ver] () at (1.5,5){\tiny{$ x^{(1)}$}};
\node[ver] () at (-1.5,3){\tiny{$ x^{(2)}$}};
\node[ver] () at (1.5,3){\tiny{$ x^{(3))}$}};
\node[ver] () at (-1.5,1){\tiny{$ x^{(4)}$}};
\node[ver] () at (1.5,1){\tiny{$ x^{(5)}$}};
\node[ver] () at (-1.5,-1){\tiny{$ x^{(6)}$}};
\node[ver] () at (1.5,-1){\tiny{$ x^{(7)}$}};
\node[ver] () at (-1.7,-3){\tiny{$ x^{(2d-2)}$}};
\node[ver] () at (1.7,-3){\tiny{$ x^{(2d-1)}$}};
\node[ver] () at (-1.5,-5){\tiny{$ x^{(2d)}$}};
\node[ver] () at (1.7,-5){\tiny{$ x^{(2d+1)}$}};

\node[ver] () at (0,5.1){\tiny{$0,1,\dots,$}};
\node[ver] () at (0,4.8){\tiny{$d-2$}};
\node[ver] () at (0,3.1){\tiny{$1,2,\dots,$}};
\node[ver] () at (0,2.8){\tiny{$d-1$}};
\node[ver] () at (0,1){\tiny{$2,3,\dots,d$}};
\node[ver] () at (0,-1){\tiny{$3,\dots,d,0$}};
\node[ver] () at (0,-2.8){\tiny{$d-1,d,$}};
\node[ver] () at (0,-3.1){\tiny{$0,\dots,d-4$}};
\node[ver] () at (0,-4.9){\tiny{$d,0,1,\dots,$}};
\node[ver] () at (0,-5.2){\tiny{$d-3$}};

\foreach \x/\y/\z in {-1.3/4/$d$,1.3/4/$d$,-1.3/2/0,1.3/2/0,-1.3/0/1,1.3/0/1,-1.5/-4/$d-2$,1.5/-4/$d-2$,-2.5/0/$d-1$,2.5/0/$d-1$}{
\node[ver] () at (\x,\y){\tiny{\z}};}
\end{scope}

\begin{scope}[shift={(8,0)}]
\foreach \x/\y/\z in {-1/5/0,1/5/1,-1/3/2,1/3/3,-1/1/4,1/1/5,-1/-1/6,1/-1/7,-1/-3/8,1/-3/9,-1/-5/10,1/-5/11}{\node[vert] (a\z) at (\x,\y){};}
\foreach \x/\y in {-1/-1.5,-1/-2,-1/-2.5,1/-1.5,1/-2,1/-2.5}{\node[extra] () at (\x,\y){};}
\foreach \x/\y in {a0/a2,a2/a4,a4/a6,a8/a10,a1/a3,a3/a5,a5/a7,a9/a11}
{\draw [edge]  (\x) -- (\y);}

\foreach \x/\y/\z in {a0/a1/4.5,a0/a1/5.5,a2/a3/2.5,a2/a3/3.5,a4/a5/.5,a4/a5/1.5,a6/a7/-.5,a6/a7/-1.5,a8/a9/-2.5,a8/a9/-3.5,a10/a11/-4.5,a10/a11/-5.5}{
\draw[edge]  plot [smooth,tension=1.5] coordinates{(\x)(0,\z)(\y)};}

\draw[edge]  plot [smooth,tension=1.5] coordinates{(a0)(-2.5,3.5)(-2.5,-4.5)(a11)};
\draw[edge]  plot [smooth,tension=1.5] coordinates{(a1)(2.5,3.5)(2.5,-4.5)(a10)};

\node[ver] () at (-1.5,5){\tiny{$ x^{(2)}$}};
\node[ver] () at (1.5,5){\tiny{$ x^{(3)}$}};
\node[ver] () at (-1.5,3){\tiny{$ x^{(4)}$}};
\node[ver] () at (1.5,3){\tiny{$ x^{(5)}$}};
\node[ver] () at (-1.5,1){\tiny{$ x^{(6)}$}};
\node[ver] () at (1.5,1){\tiny{$ x^{(7)}$}};
\node[ver] () at (-1.5,-1){\tiny{$ x^{(8)}$}};
\node[ver] () at (1.5,-1){\tiny{$ x^{(9)}$}};
\node[ver] () at (-1.5,-3){\tiny{$ x^{(2d)}$}};
\node[ver] () at (1.7,-3){\tiny{$ x^{(2d+1)}$}};
\node[ver] () at (-1.5,-5){\tiny{$ x^{(0)}$}};
\node[ver] () at (1.5,-5){\tiny{$ x^{(1)}$}};

\node[ver] () at (0,5.1){\tiny{$1,2,\dots,$}};
\node[ver] () at (0,4.8){\tiny{$d-1$}};
\node[ver] () at (0,3){\tiny{$2,3,\dots,d$}};
\node[ver] () at (0,1){\tiny{$3,\dots,d,0$}};
\node[ver] () at (0,-1){\tiny{$4,\dots,d,0,1$}};
\node[ver] () at (0,-2.9){\tiny{$d,0,1,\dots,$}};
\node[ver] () at (0,-3.2){\tiny{$d-3$}};
\node[ver] () at (0,-4.9){\tiny{$0,1,\dots,$}};
\node[ver] () at (0,-5.2){\tiny{$d-2$}};

\foreach \x/\y/\z in {-1.3/4/0,1.3/4/0,-1.3/2/1,1.3/2/1,-1.3/0/2,1.3/0/2,-1.5/-4/$d-1$,1.5/-4/$d-1$,-2.5/0/$d$,2.5/0/$d$}{
\node[ver] () at (\x,\y){\tiny{\z}};}
\end{scope}

\end{tikzpicture}
\vspace{-9mm}
\caption{The standard $2(d+1)$-vertex crystallization $(\Gamma,\gamma)$ of $\mathbb{S}^{d-1}\times \mathbb{S}^1$ (resp., $\mathbb{S}^{\hspace{.2mm}d-1} \mbox{$\times\hspace{-2.8mm}_{-}$} \, \mathbb{S}^{\hspace{.1mm}1}$) if $d$ is even (resp., odd), and  an isomorphism $I(\Gamma):\Gamma \to \Gamma$ such that $I_V(x^{(p)})=x^{(p+2)}$ (addition is modulo $2d+2$) for $p \in \{0,1,\dots,2d+1\}$ and $I_c(i)=i+1$ (addition is modulo $d+1$) for $i\in \Delta_d$.}\label{fig:SdS1:2}
\end{figure}

Now, from Remark \ref{remark:component}, if $j$ and $k$ are not consecutive colors then $\bar {\Gamma}_{\{j,k\}}$ is of type $2(d+1)C_4 \sqcup \Gamma_{\{j_1,k_1\}}\sqcup \cdots \sqcup \Gamma_{\{j_{d-2},k_{d-2}\}}$. Since $g_{\{j,k\}}=d-1$ for all $j,k$, $\bar {g}_{\{j,k\}}=2(d+1)+(d-2)(d-1)$. Therefore, we can choose a permutation $\varepsilon=(\varepsilon_0,\dots,\varepsilon_{d+1})$ such that $\bar {g}_{\{\varepsilon_{i},\varepsilon_{i+1}\}}=2(d+1)+(d-2)(d-1)$ for $0 \leq i \leq d+1$ (here the addition in the subscript of $\varepsilon$ is modulo $d+2$). Therefore, $\chi_{\varepsilon}(\bar {\Gamma})= \sum_{i \in \mathbb{Z}_{d+2}}\bar {g}_{\{\varepsilon_i,\varepsilon_{i+1}\}} -d \frac{2(d+2)(d+1)}{2}=(d+2)(2(d+1)+(d-2)(d-1))-d(d+2)(d+1)=(d+2)(2d+2+d^2-3d+2-d^2-d)=-2(d+2)(d-2)$. Thus, $\rho_{\varepsilon}(\bar {\Gamma})=1-\chi_{\varepsilon}(\bar {\Gamma})/2=1+(d^2-4)=d^2-3$. Therefore, $\mathcal{G}(M_f) \leq d^2-3$. Since $\bar \Gamma$ has $2(d+1)(d+2)$ vertices,  $\mathit{k}(M_f) \leq d^2+3d+1$. 

For $d\geq 3$, $\pi_1(M)=\mathbb{Z}$ and $\bar g_{\Delta_{d+1}\setminus \{i,j\}}=3$ for all $i,j \in \Delta_{d+1}$. Now, if the components of $\bar \Gamma_{\Delta_{d+1}\setminus\{0,1\}}$ represent generators and components of $\bar \Gamma_{\{0,1\}}$ represent relations of the fundamental group of $\pi_1(M_f)$ then, by Proposition \ref{prop:gagliardi79b}, it is easy to prove that $\pi_1(M_f)=\mathbb{Z}\times \mathbb{Z}$.

If $d=3$ then, by Lemma \ref{lemma:3}, we have $\mathcal{G}(M_f) \leq 6$. Now, from Proposition \ref{prop:lowerbound}, we have $ \mathcal G(M_f) \geq  5 rk(\pi_1(M_f))-4=5\cdot 2 -4 =6$. Therefore $\mathcal{G}(M_f) =6$. These prove the results. \hfill $\Box$

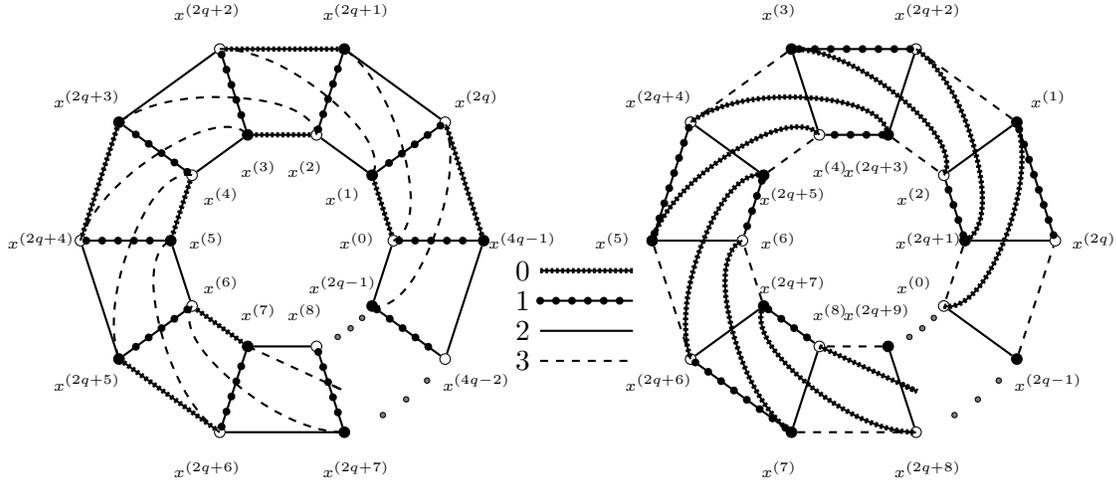
\begin{figure}[ht]
\tikzstyle{vert}=[circle, draw, fill=black!100, inner sep=0pt, minimum width=4pt]
\tikzstyle{vertex}=[circle, draw, fill=black!00, inner sep=0pt, minimum width=4pt]
\tikzstyle{ver}=[]
\tikzstyle{extra}=[circle, draw, fill=black!50, inner sep=0pt, minimum width=2pt]
\tikzstyle{edge} = [draw,thick,-]
\centering
\begin{tikzpicture}[scale=0.4]

\begin{scope}[shift={(-10,0)}]
\foreach \x/\y/\z in {0/$x^{(0)}$/b0,72/$x^{(2)}$/b2,144/$x^{(4)}$/b4,216/$x^{(6)}$/b6,288/$x^{(8)}$/b8}{
\node[ver] () at (\x:2.5){\tiny{\y}};
\node[vertex] (\y) at (\x:3.7){};
\node[ver] (\z) at (\x:5){};
} 

\foreach \x/\y/\z in
{36/$x^{(1)}$/b1,108/$x^{(3)}$/b3,180/$x^{(5)}$/b5,252/$x^{(7)}$/b7,324/$x^{(2q-1)}$/b2q}{ 
\node[ver]() at (\x:2.5){\tiny{\y}};
\node[vert] (\y) at (\x:3.7){};
\node[ver] (\z) at (\x:5){};
} 

\foreach \x/\y in {0/$x^{(4q-1)}$,72/$x^{(2q+1)}$,144/$x^{(2q+3)}$,216/$x^{(2q+5)}$,288/$x^{(2q+7)}$}{
\node[ver] () at (\x:8){\tiny{\y}};
\node[vert] (\y) at (\x:6.7){};} 

\foreach \x/\y in
{36/$x^{(2q)}$,108/$x^{(2q+2)}$,180/$x^{(2q+4)}$,252/$x^{(2q+6)}$,324/$x^{(4q-2)}$}{ 
\node[ver]() at (\x:8){\tiny{\y}};
\node[vertex] (\y) at (\x:6.7){};}

\foreach \x/\y in {308/101,300/102,316/103}{
\node[extra] (\y) at (\x:3.7){};
}
\foreach \x/\y in {308/101,300/102,316/103}{
\node[extra] (\y) at (\x:6.7){};
}
\end{scope}
\foreach \x/\y in
{$x^{(0)}$/$x^{(2q-1)}$,$x^{(0)}$/$x^{(1)}$,$x^{(2)}$/$x^{(1)}$,$x^{(2)}$/$x^{(3)}$,$x^{(4)}$/$x^{(3)}$,$x^{(4)}$/$x^{(5)}$,$x^{(6)}$/$x^{(5)}$,$x^{(6)}$/$x^{(7)}$,$x^{(8)}$/$x^{(7)}$,$x^{(4q-2)}$/$x^{(4q-1)}$,$x^{(4q-1)}$/$x^{(2q)}$,$x^{(2q)}$/$x^{(2q+1)}$,$x^{(2q+1)}$/$x^{(2q+2)}$,$x^{(2q+2)}$/$x^{(2q+3)}$,$x^{(2q+3)}$/$x^{(2q+4)}$,$x^{(2q+4)}$/$x^{(2q+5)}$,$x^{(2q+5)}$/$x^{(2q+6)}$,$x^{(2q+6)}$/$x^{(2q+7)}$}{
\path[edge] (\x) -- (\y);}

\foreach \x/\y in
{$x^{(0)}$/$x^{(1)}$,$x^{(2)}$/$x^{(3)}$,$x^{(4)}$/$x^{(5)}$,$x^{(6)}$/$x^{(7)}$,$x^{(4q-1)}$/$x^{(2q)}$,$x^{(2q+1)}$/$x^{(2q+2)}$,$x^{(2q+3)}$/$x^{(2q+4)}$,$x^{(2q+5)}$/$x^{(2q+6)}$}{
\draw [line width=2pt, line cap=rectengle, dash pattern=on 1pt off 1]  (\x) -- (\y);}

\foreach \x/\y in
{$x^{(0)}$/$x^{(4q-1)}$,$x^{(2q)}$/$x^{(1)}$,$x^{(2)}$/$x^{(2q+1)}$,$x^{(2q+2)}$/$x^{(3)}$,$x^{(4)}$/$x^{(2q+3)}$,$x^{(2q+4)}$/$x^{(5)}$,$x^{(6)}$/$x^{(2q+5)}$,$x^{(2q+6)}$/$x^{(7)}$,$x^{(8)}$/$x^{(2q+7)}$,$x^{(2q-1)}$/$x^{(4q-2)}$}{
\path[edge] (\x) -- (\y);
\draw[line width=3pt, line cap=round, dash pattern=on 0pt off 2\pgflinewidth] (\x) -- (\y);}

\foreach \x/\y/\z in
{$x^{(2q-1)}$/$x^{(2q)}$/b0,$x^{(0)}$/$x^{(2q+1)}$/b1,$x^{(2q+2)}$/$x^{(1)}$/b2,$x^{(2)}$/$x^{(2q+3)}$/b3,$x^{(2q+4)}$/$x^{(3)}$/b4,$x^{(4)}$/$x^{(2q+5)}$/b5,$x^{(2q+6)}$/$x^{(5)}$/b6,$x^{(6)}$/$x^{(2q+7)}$/b7,$x^{(7)}$/b8/b8}{
\draw[edge, dashed] plot [smooth,tension=1.5]coordinates{(\x)(\z)(\y)};}

\begin{scope}[shift={(9,0)}]
\foreach \x/\y/\z in {0/$x^{(2q+1)}$/b0,72/$x^{(2q+3)}$/b2,144/$x^{(2q+5)}$/b4,216/$x^{(2q+7)}$/b6,288/$x^{(2q+9)}$/b8}{
\node[ver] () at (\x:2.5){\tiny{\y}};
\node[vert] (\y) at (\x:3.7){};
\node[ver] (\z) at (\x:5){};
} 

\foreach \x/\y/\z in
{36/$x^{(2)}$/b1,108/$x^{(4)}$/b3,180/$x^{(6)}$/b5,252/$x^{(8)}$/b7,324/$x^{(0)}$/b2q}{ 
\node[ver]() at (\x:2.5){\tiny{\y}};
\node[vertex] (\y) at (\x:3.7){};
\node[ver] (\z) at (\x:5){};
} 

\foreach \x/\y in {0/$x^{(2q)}$,72/$x^{(2q+2)}$,144/$x^{(2q+4)}$,216/$x^{(2q+6)}$,288/$x^{(2q+8)}$}{
\node[ver] () at (\x:8){\tiny{\y}};
\node[vertex] (\y) at (\x:6.7){};} 

\foreach \x/\y in
{36/$x^{(1)}$,108/$x^{(3)}$,180/$x^{(5)}$,252/$x^{(7)}$,324/$x^{(2q-1)}$}{ 
\node[ver]() at (\x:8){\tiny{\y}};
\node[vert] (\y) at (\x:6.7){};}

\foreach \x/\y in {308/101,300/102,316/103}{
\node[extra] (\y) at (\x:3.7){};
}
\foreach \x/\y in {308/101,300/102,316/103}{
\node[extra] (\y) at (\x:6.7){};
}
\end{scope}
\foreach \x/\y in
{$x^{(2q+1)}$/$x^{(2)}$,$x^{(2q+3)}$/$x^{(4)}$,$x^{(2q+5)}$/$x^{(6)}$,$x^{(2q+7)}$/$x^{(8)}$,$x^{(2q)}$/$x^{(1)}$,$x^{(2q+2)}$/$x^{(3)}$,$x^{(2q+4)}$/$x^{(5)}$,$x^{(2q+6)}$/$x^{(7)}$}{
\path[edge] (\x) -- (\y);}
\foreach \x/\y in
{$x^{(2q+1)}$/$x^{(0)}$,$x^{(2q+3)}$/$x^{(2)}$,$x^{(2q+5)}$/$x^{(4)}$,$x^{(2q+7)}$/$x^{(6)}$,$x^{(2q+9)}$/$x^{(8)}$,$x^{(2q-1)}$/$x^{(2q)}$,$x^{(1)}$/$x^{(2q+2)}$,$x^{(3)}$/$x^{(2q+4)}$,$x^{(5)}$/$x^{(2q+6)}$,$x^{(7)}$/$x^{(2q+8)}$}{
\path[edge,dashed] (\x) -- (\y);}

\foreach \x/\y in
{$x^{(2q+1)}$/$x^{(2)}$,$x^{(2q+3)}$/$x^{(4)}$,$x^{(2q+5)}$/$x^{(6)}$,$x^{(2q+7)}$/$x^{(8)}$,$x^{(2q)}$/$x^{(1)}$,$x^{(2q+2)}$/$x^{(3)}$,$x^{(2q+4)}$/$x^{(5)}$,$x^{(2q+6)}$/$x^{(7)}$}{\draw[line width=3pt, line cap=round, dash pattern=on 0pt off 2\pgflinewidth] (\x) -- (\y);}

\foreach \x/\y in
{$x^{(2q+1)}$/$x^{(2q)}$,$x^{(1)}$/$x^{(2)}$,$x^{(2q+3)}$/$x^{(2q+2)}$,$x^{(3)}$/$x^{(4)}$,$x^{(2q+5)}$/$x^{(2q+4)}$,$x^{(5)}$/$x^{(6)}$,$x^{(2q+7)}$/$x^{(2q+6)}$,$x^{(7)}$/$x^{(8)}$,$x^{(2q+9)}$/$x^{(2q+8)}$,$x^{(0)}$/$x^{(2q-1)}$}{
\path[edge] (\x) -- (\y);}

\foreach \x/\y/\z in
{$x^{(0)}$/$x^{(1)}$/b0,$x^{(2q+1)}$/$x^{(2q+2)}$/b1,$x^{(3)}$/$x^{(2)}$/b2,$x^{(2q+3)}$/$x^{(2q+4)}$/b3,$x^{(5)}$/$x^{(4)}$/b4,$x^{(2q+5)}$/$x^{(2q+6)}$/b5,$x^{(7)}$/$x^{(6)}$/b6,$x^{(2q+7)}$/$x^{(2q+8)}$/b7,$x^{(8)}$/b8/b8}{
\draw[line width=2pt, line cap=rectengle, dash pattern=on 1pt off 1] plot [smooth,tension=1.5]coordinates{(\x)(\z)(\y)};
\draw[edge] plot [smooth,tension=1.5]coordinates{(\x)(\z)(\y)};}

 \begin{scope}[shift={(-1,-5)}]
\node[ver] (308) at (-1,4){$0$};
\node[ver] (300) at (-1,3){$1$};
\node[ver] (301) at (-1,2){$2$};
\node[ver] (302) at (-1,1){$3$};
\node[ver] (309) at (3,4){};
\node[ver] (304) at (3,3){};
\node[ver] (305) at (3,2){};
\node[ver] (306) at (3,1){};
\path[edge] (300) -- (304);
\path[edge] (308) -- (309);
\draw [line width=2pt, line cap=rectengle, dash pattern=on 1pt off 1]  (308) -- (309);
\draw [line width=3pt, line cap=round, dash pattern=on 0pt off 2\pgflinewidth]  (300) -- (304);
\path[edge] (301) -- (305);
\path[edge, dashed] (302) -- (306);
\end{scope}
\end{tikzpicture}
\caption{The standard $4q$-vertex crystallization of $L(q,1)$, and an isomorphism from the crystallization to itself.}\label{fig:Lq1}
\end{figure}

\noindent {\em Proof of Theorem} \ref{theorem:genus2}.
Let $(\Gamma,\gamma)$ be the standard $4q$-vertex crystallization of $L(q,1)$ and $I(\Gamma):=(I_V,I_c):\Gamma \to \Gamma$ be the corresponding isomorphism  as in Figure \ref{fig:Lq1}, i.e., $I_V(x^{(0)})=x^{(2q+1)}$, $I_V(x^{(1)})=x^{(2)}$, $I_V(x^{(2)})=x^{(2q+3)}$, and so on. If we choose $\tilde I(\Gamma)= I''(\Gamma) =I(\Gamma)$ as in Lemma \ref{lemma:2}, then there is a crystallization $(\bar \Gamma,\bar \gamma)$ of an orientable mapping torus $L(q,1)_f$ of a (PL) homeomorphism $f:L(q,1) \to L(q,1)$. In this case $g_{\{0,2\}}=2$ and $p=2q$. Thus, by Lemma \ref{lemma:3}, $\mathcal{G}(L(q,1)_f) \leq 1 + 5(p-g_{\{0,2\}})/2=1 + 5(2q-2)/2=5q-4$. Since $\bar \Gamma$ has $20q$ vertices, $\mathit{k}(L(q,1)_f) \leq 10q-1$.

Now, for $q=2$, $\bar g_{\{2,3,4\}}=3$. If the components of $\bar \Gamma_{\{2,3,4\}}$  represent the generators and the components of $\bar \Gamma_{\{0,1\}}$ represent the relations of $\pi_1(L(2,1)_f)$ then, by Proposition \ref{prop:gagliardi79b}, it is easy to prove that $\pi_1(L(2,1)_f)=\mathbb{Z}_2\times\mathbb{Z}$. Therefore, from Proposition \ref{prop:lowerbound}, we have $ \mathcal G(L(2,1)_f) \geq  5 rk(\pi_1(L(2,1)_f))-4=5\cdot 2 -4 =6$. Thus, from the first part, $\mathcal{G}(L(2,1)_f) =6$.

If $q=2$ then we have another isomorphism $I'(\Gamma):=(I'_V,I'_c):\Gamma \to \Gamma$
such that $I'_V$ send vertices $x^{(0)}, x^{(1)}, x^{(2)},
x^{(3)}, x^{(4)}, x^{(5)},x^{(6)},x^{(7)}$ to $x^{(2)}, x^{(5)}, x^{(0)},
x^{(7)},x^{(4)},x^{(3)},x^{(6)}$, $x^{(1)}$ respectively. Thus, by Lemma \ref{lemma:2}, there is a crystallization $(\bar \Gamma,\bar \gamma)$ of a non-orientable mapping torus $L(2,1)_{\tilde f}$ of a (PL) homeomorphism $\tilde f:L(2,1) \to L(2,1)$. In this case $g_{\{0,2\}}=2$ and $p=4$. Thus, by Lemma \ref{lemma:3}, $\mathcal{G}(L(2,1)_{\tilde f}) \leq 1 + 5(p-g_{\{0,2\}})/2=1 + 5(4-2)/2=6$. In this case, $\bar g_{\{2,3,4\}}=3$. If the components of $\bar \Gamma_{\{2,3,4\}}$  represent the generators and the components of $\bar \Gamma_{\{0,1\}}$ represent the relations of $\pi_1(L(2,1)_{\tilde f})$ then, by Proposition \ref{prop:gagliardi79b}, it is easy to prove that $\pi_1(L(2,1)_{\tilde f})=\mathbb{Z}_2\times\mathbb{Z}$. Therefore, from Proposition \ref{prop:lowerbound}, we have $ \mathcal G(L(2,1)_{\tilde f}) \geq  5 rk(\pi_1(L(2,1)_{\tilde f}))-4=5\cdot 2 -4 =6$.  Thus, we have, $\mathcal{G}(L(2,1)_{\tilde f}) =6$. These prove the results. \hfill $\Box$

\begin{example}\label{eg:RPd}
{\rm From \cite[Section 9]{ga79b}, we know the following construction of a crystallization of $\mathbb{RP}^d$ for $d\geq 2$: consider the set of the vertices  $\{(x_0,\dots,x_{d-1}) \in \mathbb{R}^d: x_i\in \mathbb{Z}_2 ~ \forall ~ 0\leq i \leq d-1\}$ of the unit cube in $\mathbb{R}^d$. Now, one can construct a $2^d$-vertex crystallization $(\Gamma,\gamma)$ of $\mathbb{RP}^d$ with the color set $\Delta_d$ by the following way.

\noindent (1) Join $x=(x_0,\dots,x_{d-1})$ and $y=(y_0,\dots,y_{d-1})$ with an edge of color $j$ if and only if $x_j \neq y_j$ and $x_0=y_0,\dots,x_{j-1}=y_{j-1},x_{j+1}=y_{j+1},\dots,x_{d-1}=y_{d-1}$.

\noindent (2) Join $x=(x_0,\dots,x_{d-1})$ and $y=(y_0,\dots,y_{d-1})$ with an edge of color $d$ if and only if $x_i \neq y_i$ for $ 0\leq i \leq d-1$.

}
\end{example}

\noindent {\em Proof of Theorem}  \ref{theorem:genus3}.
Let $(\Gamma,\gamma)$ be a crystallization of $\mathbb{S}^1 \times \mathbb{S}^1$ and $I(\Gamma):=(I_V,I_c):\Gamma \to \Gamma$ be an isomorphism as in Example \ref{example:surface}. In Example \ref{example:surface}, we have obtained a crystallization $(\bar \Gamma,\bar \gamma)$ of a mapping torus which is $T_1 \times \mathbb{S}^1$. Now, by Remark \ref{remark:isomorphism} and by the symmetry of the crystallization $\bar \Gamma$, there exist isomorphisms from $\bar \Gamma$ to itself which satisfy all the properties as in Lemma \ref{lemma:2}. Therefore, by Lemma \ref{lemma:2}, we have a crystallization of an orientable (resp., a non-orientable) mapping torus $(T_1 \times \mathbb{S}^1)_f$ (resp., $(T_1 \times \mathbb{S}^1)_{\tilde f}$) of a (PL) homeomorphism $f:T_1 \times \mathbb{S}^1 \to T_1 \times \mathbb{S}^1$ (resp., $\tilde f:T_1 \times \mathbb{S}^1 \to T_1 \times \mathbb{S}^1$). Now, from Remark \ref{remark:component}, we have $\bar \Gamma_{\{0,2\}} \cong 6 C_4$, i.e.,  $\bar g_{\{0,2\}} = 6$. If $\bar \Gamma$ is a $2\bar p$-vertex crystallization then $\bar p=12$. Let $M$ be the mapping torus $(T_1 \times \mathbb{S}^1)_f$ or $(T_1 \times \mathbb{S}^1)_{\tilde f}$. Then, by Lemma \ref{lemma:3}, we have  $\mathcal{G}(M) \leq 1+5(\bar p-\bar g_{\{0,2\}})/2=1+5\cdot 3=16$. Since the number of vertices of the crystallizations of the mapping tori is $5\cdot 4\cdot 6=120$, $\mathit{k}(M) \leq 59$. These prove part $(i)$.

From the $2^d$-vertex crystallization $(\Gamma, \gamma)$ of $\mathbb{RP}^d$ as in Example \ref{eg:RPd}, it is clear that there exist isomorphisms $I(\Gamma):=(I_V,I_c):\Gamma \to \Gamma$ which satisfy all the properties as in Lemma \ref{lemma:2}.
Since $d$ is odd, $\mathbb{RP}^d$ is orientable.
Therefore, by Lemma \ref{lemma:2}, we have a crystallization of an orientable (resp., a non-orientable) mapping torus $(\mathbb{RP}^d)_f$ (resp., $(\mathbb{RP}^d)_{\tilde f}$) of a (PL) homeomorphism $f:\mathbb{RP}^d \to \mathbb{RP}^d$ (resp., $\tilde f:\mathbb{RP}^d \to \mathbb{RP}^d$) with $(d+2)2^d$ vertices.
Since $\Gamma_{\{i,j\}}=2^{d-2}C_4$ for all $0\leq i \leq j \leq d$, from Remark \ref{remark:component}, it is clear that, if $j$ and $k$ are not consecutive colors then $\bar \Gamma_{\{j,k\}}$ is of type $(d+2)2^{d-2}C_4$. Therefore, we can choose a permutation $\varepsilon=(\varepsilon_0,\dots,\varepsilon_{d+1})$ such that $\bar g_{\{\varepsilon_{i},\varepsilon_{i+1}\}}=(d+2)2^{d-2}$ (the addition in the subscript of $\varepsilon$ is modulo $d+2$) for $0 \leq i \leq d+1$. Therefore, $\chi_{\varepsilon}(\bar \Gamma)= \sum_{i \in \mathbb{Z}_{d+2}}\bar g_{\{\varepsilon_i,\varepsilon_{i+1}\}} -d \frac{(d+2)2^d}{2}=(d+2)^2 2^{d-2}-d(d+2)2^{d-1}=(d+2)2^{d-2}(d+2-2d)=2^{d-2}(4-d^2)$. Thus, $\rho_{\varepsilon}(\bar \Gamma)=1-\chi_{\varepsilon}(\bar \Gamma)/2=1+ 2^{d-3}(d^2-4)$. Now, part (ii) follows from these.
\hfill $\Box$

\medskip
\medskip

\noindent {\em Proof of Theorem}  \ref{theorem:genus4}.
For $1\leq h \leq 2$, let $(\Gamma,\gamma)$ be a crystallization of $U_h$ and $I(\Gamma):=(I_V,I_c):\Gamma \to \Gamma$ be an isomorphism as in Example \ref{example:surface}. In Example \ref{example:surface}, we have obtained a crystallization $(\bar \Gamma,\bar \gamma)$ of a mapping torus which is $U_h \times \mathbb{S}^1$. Now, by Remark \ref{remark:isomorphism}, there exist isomorphisms from $\bar \Gamma$ to itself which satisfy all the properties as in Lemma \ref{lemma:2}. Therefore, by Lemma \ref{lemma:2}, we have a crystallization of  a non-orientable mapping torus  $(U_h\times \mathbb{S}^1)_{\tilde f}$ of a (PL) homeomorphism $\tilde f:U_h\times \mathbb{S}^1 \to U_h\times \mathbb{S}^1$. Now, the number of vertices of $\Gamma$ is $2p$, where $p=2$ (resp., $p=3$) for $h=1$ (resp., $h=2$). Thus, the number of vertices of the crystallization of the mapping torus is $5\cdot4\cdot 2p=40p$. Now, from Remark \ref{remark:component}, we have $\bar \Gamma_{\{0,2\}} \cong 2p C_4$, i.e.,  $\bar g_{\{0,2\}} = 2p$. Again we have, if $\bar \Gamma$ is a $2\bar p$-vertex crystallization then $\bar p=4p$. Now, by Lemma \ref{lemma:3}, we have  $\mathcal{G}((U_h\times \mathbb{S}^1)_{\tilde f}) \leq 1+5(\bar p-\bar g_{\{0,2\}})/2=1+5p$. Here $p=h+1$ as $p=2$ (resp., $p=3$) for $h=1$ (resp., $h=2$). Therefore, $\mathcal G(U_h\times \mathbb{S}^1)_{\tilde f})\leq 1+5p=5h+6$ and  $\mathit{k}(U_h\times \mathbb{S}^1)_{\tilde f})\leq 20p-1=20h+19$. These prove part $(i)$.

Since $d$ is even, $\mathbb{RP}^d$ is non-orientable. Therefore, we have only non-orientable mapping tori. Now, the proof of part $(ii)$ follows by the similar arguments as in the proof of part $(ii)$ of Theorem \ref{theorem:genus3}. \hfill $\Box$

\bigskip

\noindent {\bf Acknowledgement:} The author would like to thank Basudeb Datta for useful suggestions. The author is also thankful to the anonymous referees for many useful comments and suggestions. In particular, Conjecture \ref{conjecture:true} is due to one of them. The author is supported by DST INSPIRE Research Grant (DST/INSPIRE/04/2017/002471).

\end{document}